\documentclass[10pt]{article}

\usepackage{amsfonts,amsmath, amssymb}

\usepackage{xcolor}

\usepackage{enumerate}
\usepackage{epsf,epsfig,amsfonts,a4wide}
\usepackage{amsfonts, mathdots}
\usepackage{amsmath,amssymb,amsthm}
\usepackage{url}
\usepackage{amsmath,amssymb,amsthm}

\usepackage{amssymb}
\usepackage{amsthm}
\usepackage{epsf,epsfig,amsfonts,graphicx}

\newtheorem{Theorem}{Theorem}[section]
\newtheorem{Lemma}{Lemma}[section]
\newtheorem{Proposition}{Proposition}[section]

\newtheorem{Ex}{Example}[section]
\newtheorem{Remark}{Remark}[section]
\newtheorem{Fact}{Fact}[section]

\theoremstyle{remark}

\newcommand{\ti}{{\tilde\imath}}
\newcommand{\tj}{{\tilde\jmath}}

\newcommand{\LC}{\mbox{\tiny{\sf LC}}}

\newcommand{\be}{\begin{equation}}
\newcommand{\ee}{\end{equation}}

\newcommand{\R}{\mathbb{R}}
\newcommand{\Id}{\operatorname{Id}}

\newcommand{\wt}[1]{\widehat{#1}}

\def\udots{\mathinner{\mkern1mu\raise1pt\vbox{\kern7pt\hbox{.}}\mkern2mu
\raise4pt\hbox{.}\mkern2mu\raise7pt\hbox{.}\mkern1mu}}

\newcommand{\ddd}{\mathrm{d}}

\newcommand{\weg}[1]{}

\newcommand{\dd}{\operatorname{d}}

\newcommand{\tk}{{\tilde k}}

\allowdisplaybreaks

\newcommand{\const}{\mathrm{const}}
 
\title{Orthogonal separation of variables for spaces of constant curvature}

\author{Alexey V. Bolsinov\footnote{ School of Mathematics,
 Loughborough University,
 LE11 3TU, UK \ \ 
 \quad {\tt A.Bolsinov@lboro.ac.uk} } \quad
\& \quad  Andrey Yu. Konyaev\footnote{Faculty of Mechanics and Mathematics, Moscow State University, 119992, Moscow, Russia
 \ \ \quad {\tt  maodzund@yandex.ru}} \quad \& \quad Vladimir S. Matveev\footnote{
Institut f\"ur Mathematik, Friedrich Schiller Universit\"at Jena,
07737 Jena, Germany  \ \ \quad {\tt  vladimir.matveev@uni-jena.de}} 
}  
\date{}

\begin{document}

\maketitle

\begin{abstract} 
We construct all orthogonal separating coordinates in constant curvature spaces  of arbitrary signature.  Further, we construct explicit transformation between orthogonal  separating  and flat or generalised flat coordinates, as well as explicit formulas for the corresponding Killing tensors and St\"ackel matrices. 
\end{abstract} 

{\bf MSC classes:} 37J35, 70H20,  37J11, 37J06, 37J38,   70H15,  70S10,    37K06,   37K10,  37K25,     53B50,  53A20,  53B20,  53B30,  53B99    


%
%

\tableofcontents

\section{Introduction} 

Physicists  encounter separating  variables at least in their  first year  at university. Indeed, a standard way of teaching physics  is  to suggest a mathematical model that may describe  a physical phenomenon,  find (partial, approximative)  solutions to the model and compare them 
 with known  observations or  experiments. The second step of this three-step procedure, finding  and analysing   mathematical solutions,  often uses the separation of variables techniques, which in many cases allows one  to reduce  a  multidimensional system  of ODEs or  PDEs  describing the model to a system of uncoupled  one-dimensional ODEs  or even to algebraic  or integro-algebraic equations and solve or  analyse them effectively.

By {\it (orthogonal) separation of variables} on an $n$-dimensional pseudo-Riemannian manifold $(M,g)$ we understand the existence of a Killing tensor $K_{ij}$ such that the operator field $K^i_j$ has $n$ distinct eigenvalues and  the Haantjes torsion of $K^i_j$ vanishes. These  conditions  imply  the 
existence of local coordinates in which  $K^i_j$, and hence $K_{ij}$ and $g_{ij}$,  are all diagonal.
These coordinates are called {\it separating coordinates} corresponding to $K$.

Recall that a symmetric $(0, 2)$-tensor $K$ is a {\it Killing tensor} for $g$, if $\nabla_{(k} K_{ij )}=0$, where $\nabla $ is the Levi-Civita connection of $g$. Geometrically, this condition means that the function   
$$
I:TM\to \mathbb{R}\, , \ \ I(\xi)=K(\xi,\xi)  
$$ 
is constant along the orbits of the geodesic flow of $g$  (we refer to $I$ as the {\it integral  } corresponding to $K$).

Our definition is equivalent to other definitions of  orthogonal separation of variables used in the literature. 
Indeed, the existence of such a tensor $K_{ij}$  
 is equivalent to  the  (local)   existence of $n$ Killing tensors $\overset{1}K_{ij}= g_{ij},  \overset{2}K_{ij}= K_{ij}, \overset{3}K_{ij}, ..., \overset{n}K_{ij}$  such that they are linearly independent
 and such that in the tangent space at almost every point there exists a basis 
in which all tensors $ \overset{1}K_{ij},..., \overset{n}K_{ij}$ are diagonal, see e.g. \cite[Proposition 2.2]{benenti1}.  The integrals corresponding to  these tensors  Poisson commute. 
Moreover, in the separating  coordinate system $(x^1,...,x^n)$, 
the tensors  $ \overset{1}K_{ij},..., \overset{n}K_{ij}$ have  
 the so-called St\"ackel form  by \cite{eisenhart}. In other words,   there exists  a  nondegenerate $n\times n$ matrix 
$S= (S_{ij})$  with $S_{ij}$ being a function  of the $i$-th variable $x^i$ only and such that 
 the following condition holds:  
 \begin{equation} 
 \label{eq:St_intro}
S I = P,  
\end{equation} 
where  $I$ is  an $n$-vector whose components are the integrals corresponding to $ \overset{1}K,..., \overset{n}K$, 
and   $P=  (p_1^2, p_2^2,..., p_n^2)^\top$ is the $n$-vector of the squares of  momenta.  
In this case, the Hamilton-Jacobi equation 
\begin{equation} \label{eq:HJ}
 \tfrac{1}{2} g^{ij}p_ip_j = c_1 \,   ,  \  p_i= \tfrac{\partial W}{\partial x^i}  
\end{equation} 
admits (at a generic point)  a general solution of the form 
$$
W(x, c)= \sum_{i=1}^n W_i(x^i, c) \,  ,  \   c=(c_1,...,c_n) \, ,  \ \det\left[ \tfrac{\partial^2 W}{\partial x^i\partial c_j}\right]\ne 0,
$$
with  $W_i(x^i, c)=\pm \int^{x^i} \sqrt{   \sum_{s} S_{is}(\xi) c_s} \dd \xi$.  

Having such a solution $W$, one can reduce integration of the geodesic flow  to an algebraic problem: namely, in addition to the integrals $I_1,...,I_n$, the functions  $\tfrac{\partial W}{\partial c_i}$ with $i \ne 1$ 
 are constant along any solution  and, moreover, we have    $t + \const  = \tfrac{\partial W}{\partial c_1}$.   Solving this system  with respect to $x$ gives the general solution $x(t)$.  

Note, however, that in many problems  the primitives of   
$\pm  \sqrt{   \sum_{s} S_{is} c_s} $ cannot 
  be obtained explicitly so that the relations $\tfrac{\partial W}{\partial c_i}= \const_i, $ 
	$t +\const =  \tfrac{\partial W}{\partial c_1}$ form a system of integral-algebraic equations. In many cases one can still find 
	exact solutions using special functions.
	
It is known  that introducing potential energy does not pose essential difficulties: one simply replaces the vector $P=(p_1^2,...,p_n^2)^\top$ in 
\eqref{eq:St_intro} 	by  $P + F=(p_1^2+ f_1(x_1),...,p_n^2 + f_n(x^n))^\top$ with arbitrary functions $f_i$. This gives an explicit formula $I = S^{-1} (P+F)=S^{-1}P + S^{-1}F$ for  the  commuting Hamiltonians.   In particular,  $H = I_1 = \sum_{j} (S^{-1})_{1j} (p_j^2) +  \sum_{j} (S^{-1})_{1j} f_j$, where the latter term is a function on $M$ understood as a potential energy (such functions  $U(x)= \sum_{j} (S^{-1})_{1j} f_j(x^j)$ are called {\it separable potentials}).

In our paper, we  will assume that $g$ has constant sectional curvature (and  $M$ is connected). 
In this set-up, the above definition is equivalent to  what is called {\it multiplicative separation of variables}   in literature. Namely,  
 by \cite{carter, blaszak_book},  
 the  second order differential operators 
$$
\overset{\alpha}{\mathcal{K} }:= \sum_{i,j} \nabla_i \overset{\alpha}K^{ij} \nabla_j$$
mutually commute and also commute with $\Delta_g= \overset{1}{\mathcal{K}}$. In the positive definite case and on compact manifolds (or under appropriate boundary assumptions), this implies that the Helmholtz partial differential equation $\Delta_g \phi = \lambda \phi$  and time independent Schr\"odinger equation 
$(\Delta_g + U(x))  \phi = \lambda \phi$, where  $U$ is a  separable potential, 
can be reduced to a system of $n$ uncoupled 
ordinary differential equation by a {\it multiplicative}  ansatz $\phi  =\phi_1(x^1)\phi_2(x^2)\cdots \phi_n(x^n)$.  

Note that the condition that the sectional curvature of $g$ is constant appears naturally in physics, see e.g. \cite{Moon}. Clearly, the flat metric is possibly the most important one for physical applications, and indeed, separation of variables for the flat metric was used in many physical problems, see e.g. \cite{waalkens1, waalkens}. 
Metrics of nonzero sectional curvature are often used to describe phenomena near a point source (e.g., near an atomic nucleus or a star in the Universe, see e.g.  \cite{chiscop,dullin}).

An additional motivation comes from infinite dimensional  integrable systems. It was observed
 \cite{alber, blaszak_last, falqui, veselov} that  certain finite dimensional reductions  of famous integrable PDE systems   are  equivalent  (in the sense explained e.g. in \cite{blaszak_last}) to finite dimensional integrable systems  coming from separation of variables for constant curvature metrics (of possibly indefinite signature). The relation to infinite dimensional integrable systems seems to be very deep and is far from being understood. As it will be clear from the discussion below, we came to this problem studying infinite dimensional compatible Poisson brackets \cite{NijenhuisAppl2,NijenhuisAppl3,NijenhuisAppl4}. Note also that separating coordinates are orthogonal and it is known, see e.g. \cite{zakharov}, 
that orthogonal coordinates  in flat spaces are closely related to infinite dimensional  integrable systems.   There also exists a clear  relation between weakly nonlinear (=linearly degenerate) infinite dimensional  integrable systems of hydrodinamic type and orthogonal separation of variables, see e.g. \cite{NijenhuisAppl5, Ferapontov1991, KKM}.

Because of their importance,  (orthogonal)  separating coordinates for metrics of constant curvature 
have been studied since at least the 19th century. In particular, it was  known that in all dimensions and in all signatures, the  ellipsoidal  coordinates are separating for the  pseudo-Riemannian spaces of   constant curvature. It was also known that in low dimensions $n=2,3$  all separating   coordinates can be constructed from ellipsoidal  coordinates  by  passing to the limit,  see \cite{ol0,ol} and discussion in
 \cite{Kress-book}.  Note also that the orthogonality condition for separating coordinates can be weakened.  However, as shown in \cite{benenti-1} and \cite{Kress-book},    at least  in the Riemannian signature,  this more general case  reduces naturally to a description of orthogonal separating coordinates.

  In the series of fundamental works 
\cite{Kalnins1984, Kalnins-book, Kalnins1986}  E. Kalnins et al. suggested, in all dimensions and all signatures, a list of 
 separating coordinates  for  metrics of constant sectional curvature. 
 The list is parametrized
by a combinatorial object which is a graph with some numerical  labels\footnote{One needs  to invest some work in order to relate the description of separating coordinates from \cite{Kalnins1984, Kalnins-book, Kalnins1986} to a labelled  graph. More precisely,   
it was claimed for all signatures in \cite{Kalnins1984} and shown  for metrics of Riemannian signature  in \cite{Kalnins-book, Kalnins1986}
 that every separable coordinate system of a constant curvature metric can be obtained  from  an ellipsoidal  coordinate system  by  passing to the limit, and then explained how to describe this passage using a labelled graph.}. 
Moreover, in \cite{Kalnins1984}  it was claimed that this list contains {\it all}  separating coordinates. The claim was not proved in \cite{Kalnins1984}, it was merely said that the proof is similar to that in the Riemannian case. 
The special case when the metric is Riemannian was indeed  proved in \cite{Kalnins-book, Kalnins1986},  see also \cite{Kress-book}.    To the best of our knowledge, a proof in the 
general case, when the metric has indefinite signature,  did not appear in the literature.  Although it was generally believed that the description in  \cite{Kalnins1984} is correct and complete. In  particular, this is stated in \cite[Theorem 1.3]{RM2} but without proof.     More precisely, 
\cite[Theorem 1.3]{RM2}  consists of two statements. The  first statement  claims  
 that  for every separating coordinate system  on a manifold of constant curvature, 
 there  exists    a $(1,1)$-tensor $L$ which is diagonal in this coordinate system and which is geodesically compatible to the metric (see
\S \ref{sec:1.3.1} for  necessary definitions).  This statement is proved. The existence of such a tensor allows one, in principle,  to reduce the study of separating coordinate systems on $n$-dimensional manifold of constant curvature to low dimensional cases; applying this reduction  recursively (see  e.g.   \cite{RM3,CR,CR2} for examples),  one can construct all separating coordinate systems for spaces of constant curvature. However, neither comparison   with the description of  \cite{Kalnins1984}, nor explicit formulas  for the  metrics of constant curvature in separating coordinates are given in \cite{RM, RM2, RM3}. Note that the  papers \cite{CR,CR2} containing  an  explicit description assume  $n=3$. 

Our  paper fills this gap and  proves that the list from \cite{Kalnins1984} is complete   (see Theorem \ref{thm:main}).
Our description  is visually different, but equivalent to the description in \cite{Kalnins1984} and, we believe, has 
some advantages as compared to that used in \cite{Kalnins1984}   and  further publications, e.g.  \cite{Kress-book, RM, RM2, RM3}; in particular, because it is given by explicit formulas rather than by a recursive algorithm. 

We demonstrate the advantages of our approach with the following additional new results: 
	
\begin{itemize}
\item We construct explicit formulas for  flat and generalised flat coordinates as functions of separating coordinates,  see Theorems \ref{thm:casimirs}, \ref{thm:generalised_casimirs} and  \ref{thm:F_casimirs}. Previously, only special cases were known, see e.g. \cite{blaszak2007, Kalnins1984,MB}. 
	
	\item We  prove the   essential  uniqueness of a  $(1,1)$-tensor $L$ which is diagonal in the separating coordinate system   and is geodesically compatible with the metric (see Theorem \ref{thm:L} and necessary definitions  in \S \ref{sec:1.3.1}). 
	Such $(1,1)$-tensors  play  a key role in the algebraic approach to classification of separating  coordinates for constant curvature spaces, see e.g. \cite{schoebel-1, schoebel0, schoebel1, schoebel2}, and also in the approach of \cite{RM,RM2, RM3, CR, CR2}. We use the essential uniqueness  of $L$   when discussing those cases where two sets  of parameters describe equivalent separating coordinates, see Theorem   \ref{thm:2}. 
	
	\item We construct Killing tensors  and the corresponding  St\"ackel matrix (Theorems \ref{thm:3} and  \ref{thm:st}). These tasks were initiated  in  \cite[p. 203]{Kalnins1984} (but postponed for publication elsewhere).
\end{itemize}

In our proof, we  show that the existence of separating coordinates can be reduced  to a system of PDEs which was  studied in our recent paper \cite{NijenhuisAppl3}.   The motivation of \cite{NijenhuisAppl3} is very different from that of \cite{Kalnins1984, Kalnins1986}. The  goal of 
 \cite{NijenhuisAppl3}  was to  describe  all compatible pencils of $\infty$-dimensional geometric Poisson structures  
of the form  $\mathcal{P}_3+ \mathcal{P}_1$  on the loop space, where $ \mathcal{P}_i$ has order $i$ and $\mathcal{P}_3$ is Darboux-Poisson. We achieved this goal under natural nondegeneracy assumptions and  
obtained a full list of such structures; this list is also parametrized by a  graph with labels.
 Then, we observed that our graph with labels and the description  
 obtained in   \cite{Kalnins1984}, although visually different,   are actually combinatorially equivalent. This suggested that one could  use the calculation  and arguments  of   \cite{NijenhuisAppl3} to  prove the claim of \cite{Kalnins1984}, and this is how we proceed in the current paper. First we explain how to reduce the system of PDEs describing the existence of separating coordinates to those equations which were studied and completely solved in \cite{NijenhuisAppl3}. Then we use this solution 
 to prove the claim of  \cite{Kalnins1984}.  Of course,  a direct
 proof, without using  \cite{NijenhuisAppl3}  but repeating all the steps from \cite{NijenhuisAppl3}, is also possible and we explain how it goes. 

 Let us emphasize  that  certain steps  of  the proof in the 
 Riemannian case   from \cite{Kalnins-book,Kalnins1986}, in combination with \cite{eisenhart} (or Appendix of  \cite{Kalnins-book}),  resemble  those from \cite{NijenhuisAppl3}.  This makes us believe  that  E. Kalnins, W. Miller and G. Reid, the authors of  \cite{Kalnins1984}, had the proof of their claim,  but did not publish it   due to   its length and complexity.  In fact,  in 
the case of indefinite  signature new phenomena appear and the proof becomes  more complicated 
  compared  to the Riemannian case. Also note  that  the proof from \cite{NijenhuisAppl3} is quite long and complex.

Recall  that by \cite{hammerl,kruglikov}, for a real-analytic metric and therefore for the metrics  of constant curvature, any Killing tensor is also real-analytic. Therefore, a local existence of separating coordinates   implies their 
 existence near almost every point; moreover, 
separating coordinates can be chosen to be real-analytic.

Let us also note that we allow some coordinates to be complex-valued; the corresponding momenta are also  complex-valued. We assume, of course, that along with a complex coordinate $z^j= x^j +  i x^{j+1}$,  
its  complex-conjugate  $\bar z^j= x^j -  i x^{j+1}$ is also a coordinate. Having formulas for the metric,  Killing tensor, or separating ansatz  in coordinates involving complex conjugate pairs $z^j, \bar z^j$,
it is straightforward to rewrite them in real-valued  coordinates (see e.g. \cite{DR}).

\subsection{Description of separating coordinates for metrics of constant curvature} \label{sec:2} 

We start with an explicit construction of  a  family of diagonal 
metrics in some coordinates.  Our main result, Theorem \ref{thm:main}, states that these coordinates are separating, the metrics have constant curvature, and  any pair (metric of constant curvature, separating coordinate system)
  is contained in the family, modulo renumeration of coordinates and coordinate changes of  
the form  $x_{\mathrm{\mathrm{new}}}^i= x_{\mathrm{\mathrm{new}}}^i(x_{\mathrm{\mathrm{old}}}^i)$.

Each metric from the family  is built based  on the following data:

\begin{enumerate} \item Natural number $B\le n:= \dim M$  \ (=``number of blocks''). 
\item Natural numbers $n_1,...,n_B$ (=``dimensions of blocks'') with $\sum_{\alpha=1}^Bn_\alpha= n$.
\item In-directed rooted forest $\mathsf F$ (with $B$ vertices which we denote  by numbers $1,...,B$), that is,  
an oriented graph such that each of its connected 
component is a rooted tree and for any vertex there exists a necessarily unique 
 oriented path towards a  root.
The edge of the graph connecting vertices $\beta$
  and $\alpha$ and oriented towards $\alpha$ will be denoted by $\vec{\beta \alpha}$. 
 An example of a rooted tree  with $B=4$ vertices is on 
Figure  \ref{Fig:1}.  

\begin{figure}
\begin{center}
 \includegraphics[width=\textwidth]{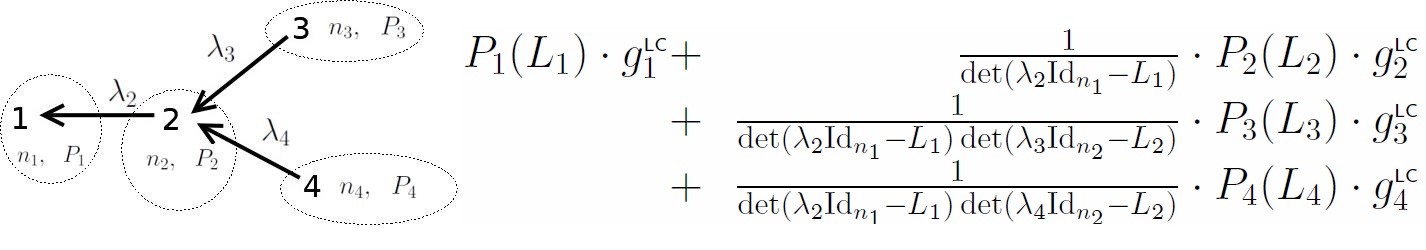} 
\end{center}
\caption{An example of an  in-tree, structure of  its labels and the form of the corresponding metric \eqref{eq:hatg}.
 \label{Fig:1}}
\end{figure}

\item Each edge $\vec{\beta \alpha}$   of $\mathsf F$  is labelled by a  number $\lambda_\beta$.  

\item Polynomials $P_\alpha,$ $\alpha  = 1,...,B$ of degree  at most $n_\alpha+1$, whose coefficients are denoted as follows:
$$P_\alpha(t) =   \overset{\alpha}{a}_0 + \overset{\alpha}{a}_1  t +...+  \overset{\alpha}{a}_{n_\alpha+1}t^{n_\alpha+1}.$$
\end{enumerate}

The structure of an in-directed forest   defines a natural strict partial\weg{\footnote{ It is convenient, though not necessary, to assume that the vertices of $\mathsf F$ are labelled in such a way that $\alpha \prec \beta$ implies $\alpha<\beta$}} order (denoted by $\prec$) on the set $\{1,...,B\}$: 
for two numbers $\alpha\ne \beta \in \{1,...,B\}$ we set $\alpha \prec \beta$, if there exists an oriented path from $\beta$ to $\alpha$.
For instance,  for the graph shown on Figure \ref{Fig:1}, we have $1\prec2 \prec  3$, $1 \prec 2\prec 4$. 
We say $\alpha= \operatorname{next}(\beta)$, if  the graph contains the edge $\vec{\beta \alpha}$.   
In Figure  \ref{Fig:1},  the root is $1$,   the leaves are $3$ and $4$ and we have:  $1= \textrm{next}(2)$, $2= \textrm{next}(3)$, $2= \textrm{next}(4)$.

Further, we assume that the polynomials $P_\alpha$ and  numbers $\lambda_\beta$     satisfy the following restrictions:
\begin{itemize}
\item[(i)] If $\mathsf F$  has  more than one connected component and therefore more that one root,
 then for any root $\alpha$ we have  $\overset{\alpha}{a}_{n_\alpha+1}=0$,  i.e., $\deg P_\alpha\le n_\alpha$.

\item[(ii)] If $\alpha= \operatorname{next}(\beta)$, 
then $\lambda_\beta$ is a root of $P_\alpha$ and $\overset{\beta}{a}_{n_\beta+1}= P_\alpha'(\lambda_\beta)$, where $P'(t)$ denotes the derivative of $P(t)$.

\item[(iii)] If  for at least two vertices $\beta\ne \gamma$ we have 
$ \operatorname{next}(\beta)=  \operatorname{next}(\gamma)=: \alpha$ with $\lambda_\beta = \lambda_\gamma=:\lambda$,   then $\lambda$ is a double root of $P_\alpha$. Note that in view of (ii) this automatically implies 
$\overset{\beta}{a}_{n_\alpha+1} = \overset{\gamma}{a}_{n_\gamma+1} =0$.
\end{itemize}

Based on these data, we construct the following  diagonal metric  $g$. 
We first divide our diagonal  coordinates into $B$ blocks of  dimensions $n_1,...,n_B$ with   $n_1+...+n_B=n$:
	\begin{equation}
	\label{eq:decomposition}
	(\underbrace{x_1^1,...,x_1^{n_1}}_{X_1},...,\underbrace{x_B^{1},...,x_B^{n_B}}_{X_B}).
	\end{equation}

Next, for every $\alpha= 1,...,B$ we consider the  $n_\alpha$-dimensional  contravariant 
metric $g^{\LC}_\alpha$ and $n_\alpha$-dimensional operator  (= $(1,1)$-tensor) 
$L_\alpha$ given by:
\begin{equation}
\label{eq:tildeg}
	g^{\LC}_\alpha= \sum_{s=1}^{n_\alpha} \left(\prod_{j\ne s} (x_\alpha^s- x_\alpha^j)\right)^{-1}  \left(\tfrac{\partial}{\partial x_\alpha^s}\right)^2\  \ , \ \  L_\alpha= \textrm{diag}(x_\alpha^1,...,x_\alpha^{n_\alpha}).
\end{equation}

Finally, we introduce a diagonal  metric 
\begin{equation}
\label{eq:hatg}
  g = \operatorname{diag}( g_1, \dots,  g_B)\quad \mbox{with }  g_\alpha = 
f_\alpha \cdot P_\alpha(L_\alpha) \cdot g^{\LC}_\alpha,
\end{equation}
and  $f_\alpha$ given by  
\begin{equation} 
\label{eq:falpha}
f_\alpha = \prod_{\begin{array}{c}s\stackrel{\prec}{=} \alpha \\
 \textrm{$s$ not a root} \end{array}} \frac{1}{\det(\lambda_{s}\cdot\Id_{n_s} -   L_{\operatorname{next}(s)})}, 
\end{equation} 
where $s\stackrel{\prec}{=} \alpha$ means that $s\prec \alpha$ or $s=\alpha$ and $\Id_{n_s}$ denotes the identity operator in dimension $n_s$. If $\alpha$ is a root, we set $f_\alpha=1$.  We also use $P_\alpha(L_\alpha)$ for the polynomial $P_\alpha$ applied to the operator $L_\alpha$. Since $L_\alpha=\operatorname{diag}(x_\alpha^1,..., x_\alpha^{n_\alpha})$, 
we simply have $P_\alpha(L_\alpha)=\operatorname{diag}\bigl(P_\alpha(x_\alpha^1),..., P_\alpha(x_\alpha^{n_\alpha})\bigr)$.  Similarly, 
$P_\alpha(L_\alpha) \cdot g^{\LC}_\alpha$ is the product of the diagonal matrices $P_\alpha(L_\alpha)$ and $ g^{\LC}_\alpha$. 
See  Figure \ref{Fig:1} for an example. See also \cite[first formula for $g$ in Example 4.1]{NijenhuisAppl3} for an example with flat $g$ and 
  $B=n_1=n_2=2$.

This completes the description of  the family  of (contravariant) metrics and we can state our main result.

\begin{Theorem} 
\label{thm:main}
The metrics \eqref{eq:hatg}   have constant sectional curvature and  $x^1,...,x^n$ are separating coordinates for them.  Moreover,  
 every  pair (metric of constant curvature, separating coordinates for it) 
 can be brought to this form  by renumeration  of coordinates and coordinate transformations  of the form $ x_{\mathrm{new}}^i= x_{\mathrm{new}}^i(x_{\mathrm{old}}^i)$, $i=1,\dots,n$.
\end{Theorem}

The curvature of the metric $g$ given by \eqref{eq:hatg}  is $-\tfrac{1}{4} \overset{\alpha}a_{n_\alpha+1}$, where $\alpha$ is a root of the forest $\mathsf F$. Note that if $\mathsf F$ has several connected components, then the curvature is  zero by {\rm (i)}. In this case,   $g$ is the direct product of the metrics corresponding to the connected components.

\subsection{What  parameters describe equivalent separating coordinates?}\label{sec:1.2}

Let us now discuss    when two different sets of parameters  (in-directed rooted forest $\mathsf F$, numbers $B$, $n_\alpha$, $\lambda_\beta$,  and polynomials 
 $P_\alpha$ satisfying conditions {\rm (i)--(iii)} from \S \ref{sec:2})
describe  
the same  separating coordinates  (of a fixed metric of constant curvature) up to the following freedom: re-numeration of  coordinates and  coordinate  changes of the form 
$x^i_{\mathrm{new}}= x^{i}_{\mathrm{new}}(x^i_{\mathrm{old}})$.  \weg{That is, when there exists a isometry of a metric which sends one system of separating coordinates to another system, modulo re-numeration  and  coordinate  changes  
$x^i_{\mathrm{new}}= x^{i}_{\mathrm{new}}(x^i_{\mathrm{old}})$. }

Clearly, if one   set of data can be obtained from another  by the operations described below, then the corresponding separating coordinates, in the above sense, coincide:

\begin{itemize} \item[(A)]  For some $\alpha \in \{1,...,B\}$, 
 we re-numerate (via a bijection from $\{1,...,n_\alpha\}$ to itself) the coordinates in the coordinate block $X_\alpha=(x_\alpha^1,...,x_\alpha^{n_\alpha})$.

 \item[(B)]  For some $\alpha \in \{1,...,B\}$, 
in the coordinate block $X_\alpha=(x_\alpha^1,...,x_\alpha^{n_\alpha})$,   we change the variables  by the formula $
{\overset{\mathrm{new}}x}{}_\alpha^i= A{\overset{\mathrm{old}}x}{}_\alpha^i + C$ with constants $A\ne 0$ and $C$, and ``compensate'' this as follows:
\begin{itemize}

\item Change  the polynomial $P_\alpha(t)$ by the polynomial  $A^{n_\alpha+1}P_\alpha\left(\tfrac{t- C}{A}\right)$. 

\item  
  For each $\beta$ such that $\operatorname{next}(\beta)= \alpha$, change the number $\lambda_\beta$ by $A \lambda_\beta+ C$. 

\item    For each $\gamma$ such that the oriented path from $\gamma$  to the  root passes through $\alpha$, multiply the polynomial 
$P_\gamma$ by $A^{n_\alpha}$.

\end{itemize}

\item[(C)]  Rename  the vertices in the in-directed forest  (via a bijection from $\{1,...,B\}$ to itself).

	\item[(D)] If $\alpha$ is a leaf (= no incoming edges) and $n_\alpha=1$,    replace 
	the polynomial $P_\alpha$ by any other polynomial $\tilde P_\alpha$ satisfying conditions {\rm (i)--(iii)} such  that the sign of   $P_\alpha(x_{old})$  coincides with that of  
	$\tilde P_\alpha(x_{new})$.  
\end{itemize}

It is easy to check that the operations do not destroy 
 conditions {\rm (i)--(iii)}
of  \S \ref{sec:2}.    The next theorem says that these operations exhaust all possibilities.

\begin{Theorem} \label{thm:2} 
Suppose two sets of admissible data $($in-directed rooted forest $\mathsf F$, numbers $B$, $n_\alpha$, $\lambda_\beta$, and polynomials $P_\alpha$ satisfying conditions {\rm (i)--(iii)} from {\rm\S \ref{sec:2}}$)$
describe  the same  separating coordinates up to re-numeration of coordinates and transformations of the form $x^i_{\mathrm{new}}= x^{i}_{\mathrm{new}}(x^i_{\mathrm{old}})$.
   Then one of the sets of admissible data  can be transformed into  the 
other one using a finite number of operations {\rm (A), (B), (C), (D)} described above.
\end{Theorem}

\begin{Remark}
\label{rem:1.1}{\rm 
One can use operations (B) and (D)  to reduce the number of parameters in the labelling of  $\mathsf F$. 
 }\end{Remark}

\subsection{Killing tensors in separating coordinates and the corresponding  St\"ackel matrix} \label{sec:killing}

The existence of  Killing tensors for metric \eqref{eq:hatg} that are diagonal in the coordinates $(x^1,...,x^n)$ will  be clear from the proof. \weg{ In \S \ref{sec:2.1} we recall the 
 equations on the diagonal entries of $g$  which are equivalent to the fact that $(x^1,...,x^n)$ are separating for $g$. Formula   \eqref{eq:hatg}  was derived from these equations (using the additional assumption that $g$ has constant 
curvature and relation  \cite[(57)]{NijenhuisAppl3})  and the entries of \eqref{eq:hatg} indeed satisfy them.} 
The goal of this section is to give an explicit formula for  these Killing tensors.

\weg{They will be constructed block-wise, with
  block number $\alpha$ giving   an $n_\alpha$-dimensional space of Killing tensors  such that 
the corresponding $(1,1)$-tensors $\overset{\alpha}K{}^i_j$ have generically $n_\alpha$ different eigenvalues inside the block  number $\alpha$ and  
  are proportional to the identity (with  nonconstant coefficients)   in   other blocks. }

As a by-product (Theorem \ref{thm:st})   we obtain  an explicit formula   for the St\"ackel matrix. This  result is important from the point of view of applications: indeed, as recalled  in  Introduction,   the solution $W(x,c)$ of the  Hamilton-Jacobi equation  corresponding to separating coordinates involves   the St\"ackel  matrix.  Note also that in \cite[p. 203]{Kalnins1984} Kalnins et al. explicitly asked  to construct the St\"ackel matrices corresponding to separating variables on spaces of constant curvature.

Following \cite{benenti,hyperbolic} (see also Facts \ref{fact:1}, \ref{fact:2}  from \S \ref{sec:1.3.1}), we recall the construction of Killing tensors for the Levi-Civita metric \eqref{eq:tildeg}.  Consider 
the following   family $\overset{\alpha}S_{\LC}(t_\alpha)$ 
 of $(1,1)$ tensors on the $\alpha$-block with coordinates $X_\alpha=(x_\alpha^1,...,x_\alpha^{n_\alpha})$:
\begin{equation}
\label{eq:S} 
\overset{\alpha} S_{\LC}(t_\alpha) = \left(t_\alpha \Id_{n_\alpha}-L_\alpha  \right)^{-1} \det\left( t_\alpha \Id_{n_\alpha}- L_\alpha\right). 
\end{equation}
Note that  $\overset{\alpha}S_{\LC}(t_\alpha)$
 is polynomial of degree $n_\alpha-1$ in $t_\alpha$. 
 It is known, see e.g. \cite{benenti}, that for every 
$t_\alpha$  the $(0,2)$-tensor $\overset{\alpha} K{}^{\LC}_{ij}$ given, for any two tangent vectors $\xi, \eta$,  by 
\begin{equation} 
\label{eq:K} 
\overset{\alpha}K{}^{\LC}(\xi  ,\eta )=   {g}^{\LC}_\alpha  ( \overset{\alpha} S_{\LC}(t_\alpha) \xi, \eta )
  \end{equation}  
 is a Killing tensor for ${g}^{\LC}_\alpha$. Clearly, the number of linearly independent Killing tensors within the family \eqref{eq:K} is $n_\alpha$ and generically they have $n_\alpha$ different eigenvalues.

Next, with the help of 
   the  family of  Killing tensors $\overset{\alpha}K{}^{\LC}$ of the $n_\alpha$ dimensional metric $g_\alpha^{\LC} $,
	  let us construct Killing tensors for the metric $g=\operatorname{diag}(g_1,\dots, g_B)$ given by \eqref{eq:hatg}. We first  define the (1,1)-tensor $\overset{\alpha}S(t_\alpha)$  to be
 block-diagonal in the coordinates $(x^1,...,x^n)=(X_1,...,X_B)$ with the following blocks:  
 \begin{itemize}
  \item the $\alpha$-block  ($\alpha \in \{1,\dots, B\}$) is  $\overset{\alpha}S_{\LC}/f_\alpha$, where  $\overset{\alpha}S_{\LC}$ is   given by \eqref{eq:S} and $f_\alpha$ by \eqref{eq:falpha}; 
  \item   every $\beta$-block with $\alpha\prec \beta$ is  given by 
$$
 \overset{\alpha\beta}\Phi(t_\alpha)\Id_{n_\beta},
$$
   where    $\overset{\alpha\beta}\Phi $ is  the  polynomial in  $t_\alpha$ of degree $n_\alpha-1$
  given by 
	\begin{equation} \label{eq:f} 
\overset{\alpha\beta}\Phi(t_\alpha):= \tfrac{1}{t_\alpha-\lambda_\gamma } \tfrac{1}{f_\alpha}\bigl(\det(t_\alpha \Id_{n_\alpha}-L_\alpha  )- \det(\lambda_\gamma \Id_{n_\alpha}-L_\alpha )\bigr),  
	\end{equation} 
	where $\gamma$ is such that $\gamma\prec \beta$ or $\gamma=\beta$ and  $\operatorname{next}(\gamma)= \alpha$  (these conditions  determine $\gamma$ uniquely\footnote{In other words,  $\gamma$  is the vertex on the unique oriented path 
		from $\beta$ to $\alpha$ in $\mathsf F$ that precedes $\alpha$.};
  
  \item for every other $\mu\in  \{1,\dots, B\}$,    the corresponding $\mu$-block  is zero. 
\end{itemize}

\begin{Theorem}\label{thm:3}  
For every $\alpha$  and $t_\alpha$ the $(0,2)$-tensor $\overset{\alpha}K(t_\alpha)$ given, 
for any two tangent vectors $\xi, \eta$,   by 
\begin{equation} \overset{\alpha}K (t_\alpha)(\xi  ,\eta )= g ( \overset{\alpha} S(t_\alpha) \xi, \eta )
 \label{eq:K1} \end{equation}  
is a  Killing tensor for $g$. Moreover, 
for any $\alpha, \beta$ and any $t_\alpha$, $t_\beta$ the integrals corresponding to the Killing tensors  $\overset{\alpha}K(t_\alpha)$ and $\overset{\beta}K(t_\beta)$  Poisson-commute.  
\end{Theorem} 

Clearly,  the space spanned by each individual family  $\overset{\alpha}K(t_\alpha)$ is $n_\alpha$-dimensional, and the spaces spanned by different families  $\overset{\alpha_i}K (t_{\alpha_i})$ are mutually independent, so that  they  span a vector space of 
 dimension $n_1+...+n_B=n$.

Let us now describe the integrals from Theorem \ref{thm:3} by a St\"ackel matrix, that is,  find an $n\times n$ matrix $S$ such that 
the  integrals corresponding to the Killing tensors are given   by the  ``St\"ackel''  formula  
\begin{equation} \label{eq:St}
S I = P, 
\end{equation} 
where
$I=(I_1,..., I_n)^\top$ is  the $n$-vector whose components are integrals corresponding to the Killing tensors
and $P=  (p_1^2,..., p_n^2)^\top$ is the vector of squares of  momenta.    
Recall  that
the  $(i,j)$-component $S_{ij}$  is a function of  $x^i$  only and  this property implies Poisson commutativity of the components of  $I$.

As specific Killing tensors related to the integrals $I_1,\dots, I_n$, we take the coefficients of the polynomials $\overset{\alpha}K(t_\alpha)$.  Namely,  
the first Killing tensor is the coefficient of $\overset{1}K(t_1)$  at $t_1^{n_1-1}$, the second Killing tensor is the coefficient of $\overset{1}K(t_1)$  at $t_1^{n_1-2}$, \ldots, the $n_1$-th Killing tensor is the free term  of $\overset{1}K(t_1)$, the $(n_1+1)$-st Killing tensor is the coefficient of $\overset{2}K(t_2)$  at $t_2^{n_2-2}$ and so on.

We will see  that the part of the matrix $S$ which is relevant for the integrals coming from the family $\overset{\alpha}K(t_\alpha)$ corresponds to the coordinate block $X_\alpha$, that is,  contains the rows from $n_1+...+n_{\alpha-1}
+1$ to $n_1+...+n_{\alpha}.$ This part of the matrix is naturally divided into  $B$ 
blocks of dimensions $n_1\times n_\alpha$, $n_2\times n_\alpha$,..., $n_B\times n_\alpha$.  The following theorem describes them.

\begin{Theorem} \label{thm:st} The block number $\alpha$ is a square $n_\alpha\times n_\alpha$ matrix 
 whose $(i,j)$ element is given by $\frac{1}{ P_\alpha(x_\alpha^i)}
\left(x_\alpha^{i}\right)^{n_\alpha -j}$. 
The block number $\beta$ with $\beta\ne \alpha$ is given as follows: 
\begin{itemize}
\item it is zero if $\beta\ne \operatorname{next}( \alpha)$; 
\item for
 $\beta = \operatorname{next}( \alpha)$ the $i$-th  component of the 
first column of the block is 
$\frac{1}{P_\beta(x_\beta^i)(x_\beta-\lambda_\beta )}$ and  all the other columns are zero.
\end{itemize}
\end{Theorem}

For example, for ${\mathsf F} = \left(1 \overset{\lambda_2} \longleftarrow 2\right)$  with two  blocks  of dimensions $n_1=3$ and $n_2=2$, the matrix $S$ is the product of two matrices: 
$$   
 S= \left(\begin{array}{ccccc}
P_1(x^{1}) & 0 & 0 &0 & 0 
\\
0 & P_1(x^{2}) & 0 & 0 & 0 
\\
0 & 0 & P_{1}(x^3) & 0 & 0 
\\
 0 & 0 & 0 & P_2(x^{4}) & 0 
\\
 0 & 0 & 0 & 0 & P_2(x^5)
\end{array}\right)^{-1}\left(\begin{array}{ccccc}
(x^{1})^{2} & x^{1} & 1 & \frac{1}{x^{1}-\lambda_2} & 0 
\\
 (x^{2})^{2} & x^{2} & 1 & \frac{1}{x^{2}-\lambda_2} & 0 
\\
 (x^{3})^{2} & x^{3} & 1 & \frac{1}{x^{3}-\lambda_2} & 0 
\\
 0 & 0 & 0 & x^{4} & 1 
\\
 0 & 0 & 0 & x^{5} & 1 
\end{array}\right)
.$$

The proof of Theorem \ref{thm:st} is an exercise in linear algebra and  is  left to the reader: one needs to invert the matrix $S$ from Theorem \ref{thm:st} and check that  
the  rows give coefficients of the Killing tensors from Theorem \ref{thm:3}.

\subsection{Transformation between  separating and  (generalised)  flat  coordinates} \label{sec:flat}

We consider a metric $g$ given by  \eqref{eq:hatg} in  coordinates $x^1,...,x^n$ and assume first 
that it is flat. 
Following terminology in \cite{flat}, we  say that local functions $Y^1(x^1,...,x^n),...,Y^n(x^1,...,x^n )$ form a  {\it flat coordinate system} for $g$,
 if they are linearly independent and for every $i$  we have $\nabla \nabla Y^i =0$, 
where $\nabla $ is the Levi-Civita connection of $g$.  

In the case of a metric $g$ of constant nonzero curvature $K$,
by {\it generalised flat coordinates}  we understand $n+1$ linearly independent local  functions $Y^1,...,Y^{n+1}$ such that for every $i\in \{1,..., n+1\}$  we have 
\begin{equation} \label{eq:vBn}
\nabla \nabla Y^i + K Y^i g= 0.
\end{equation} 

In order to  explain  this terminology, for a metric  $g=g_{ij}$ on $M$  
 consider the cone  $\wt M:= \mathbb{R}_{>0}\times M$ with the (covariant) metric $\wt g = \tfrac{1}{K} (\dd r)^2 + r^2 \sum_{i,j=1}^n g_{ij} \dd x^i \dd x^j$, where $r$ is the coordinate on  $\mathbb{R}_{>0}$.  If $g$ has constant  curvature $K$, then  the  metric $\wt g$  is flat.  
It is known and easy to check that 
the function $Y$ on $M$ satisfies \eqref{eq:vBn} if and only if the function $\wt Y:= r Y$ on $\wt M$ satisfies $\wt \nabla  \wt \nabla \wt Y =0$. In other words,   $n+1$ functions  $Y^1,...,Y^{n+1}$ on $(M,g)$ form  generalised flat coordinates,  if  and only if the functions   $\wt Y^1:= r Y^1,...,  \wt Y^{n+1} := r Y^{n+1}$ form  flat coordinates on $(\wt M,\wt g)$. 

\begin{Remark}
\label{rem:1.2}{\rm 
Equivalently,   (generalised) flat coordinates $Y^1, Y^2, ...$ for a (contravariant) metric $g$  
of curvature $K$ can be characterised by the condition
$$
g^*(\dd Y^i, \dd Y^j) = G^{ij} - K Y^i Y^j,
$$
where $G=\Bigl( G^{ij} \Bigr)$ is a constant non-degenerate matrix  of size $n\times n$, if $g$ is flat, and $(n+1)\times(n+1)$, if $K\ne 0$.    In the latter case,  $G^{ij} = \wt g^{ij}$, where  $\wt g$  is the cone metric in the coordinates $rY^1,...,rY^{n+1}$.
}\end{Remark}

The importance of flat and generalised  flat coordinates is clear. In mathematical physics, most problem are initially given in a  flat or generalised flat coordinate system. Separation of variables method is used to  find   solutions of the corresponding  ODEs/PDEs which are naturally given in separating coordinates. In order to compare them with observations or an experiment, one  typically should transform them  to the
initial (generalised) flat coordinate system.

\begin{Remark} \label{rem:new}  {\rm 
We will also use the existence of (generalised) flat coordinates in the proof of  Theorem \ref{thm:main}. Namely, the existence of flat or generalised flat coordinates implies that the metrics given by  \eqref{eq:hatg} have constant curvature which is the first statement of Theorem \ref{thm:main}.
}\end{Remark}

In this  section we describe  flat and generalised flat  coordinate system  for the  metrics $g$ given by \eqref{eq:hatg} together with the corresponding matrix $G^{ij}$ (see Remark \ref{rem:1.2}). 
   We start with the case when $\mathsf F$  contains only one vertex, so that  the (contravariant) metric $g$ has the form   $P(L)g_{\mathsf{\mathsf{LC}}}$ with a polynomial $P$ of degree at most $n+1$.

\subsubsection{Flat and generalised flat coordinates   for the metric $P(L)g_{\mathsf{LC}}$} \label{sec:flatLC}

We start with  the flat case, when $P$ has degree at most $n$. In this case,  $g=P(L)g_{\mathsf{LC}}$ is flat.

\begin{Theorem} \label{thm:casimirs}
Assume 
  $P(t)=-4\, c\, (t-\mu_1)^{k_1}(t-\mu_2)^{k_2}...(t-\mu_m)^{k_m}$ with $d:= n-(k_1+...+k_m)\ge 0$ and mutually different $\mu_i$.
Then, for every $i \in \{1,...,m\}$  and   $\alpha\in \{0,...,k_{i}-1\}$ the   
function  \begin{equation} \label{eq:10}
Y_i^{\alpha}:= \tfrac{1}{\alpha!}\tfrac{\dd^\alpha}{\dd t^\alpha }  \sqrt{\det( t\cdot {\Id} -L)}_{|t=\mu_i}    \end{equation}
has  the property $\nabla \nabla X_i^{\alpha}=0$. 
Moreover, the  functions 

\begin{equation}  \label{eq:11}
Y_{\infty}^{\alpha}:= \tfrac{1}{\alpha!}\tfrac{\dd^\alpha}{\dd t^\alpha} \sqrt{\det(\Id - t\cdot {L})}_{|t=0}    - 
R^{(\alpha)}   
\end{equation}
with $\alpha=1,...,d$ also have the property $\nabla \nabla Y_\infty^{\alpha}=0$.    The constants $R^\alpha$ are defined as follows
\begin{equation}
\label{eq:bols_5}
R^{(\alpha)} = \tfrac{1}{\alpha!} \tfrac{\dd^\alpha}{\dd t^\alpha}  \prod_{i=1}^m (1 - t\mu_i)^{k_i}{}_{|t=0}.
\end{equation}

These $n$ functions $$Y^1:= Y_1^0,..., Y^{k_1}:= Y_1^{k_1-1}, Y^{k_1+1}:= Y_2^0,...,Y^{k_1 +k_2}:= Y_2^{k_2-1},...,Y^{n-d+1}:= Y_\infty^1,...,
Y^n:= Y_\infty^d$$
are linearly independent and form flat coordinates for the contravariant metric $g=P(L)g_{\mathsf{LC}}$.
\end{Theorem}

According to our literature research, the following special cases of Theorem \ref{thm:casimirs} were known. The  classically known case is  when  the polynomial $P$ has only simple roots (see Example \ref{ex:riem1}). The case  $P(t)= t^k $, $k=0,...,n$ 
(see Example \ref{ex:riemb})   is solved  in 
\cite{blaszak2007}. 
  A discussion of the  general case from the viewpoint of the limit procedure can be found in \cite{Kalnins1984} and in particular, the  case when  $P$ has one double root and other roots are simple is given there with explicit formulas.

\begin{Remark}\label{rem:matrix}{\rm
Calculations in the proof of Theorem \ref{thm:casimirs} also give  us the matrix of the contravariant metric $g = P(L)g_{\mathsf{LC}}$ in coordinates $Y^1,...,Y^n$.  
It is blockdiagonal with $m+1$  blocks of dimensions $k_1,...,k_m,d$. The $i$-th block 
 with $i= 1,...,m$ is  the product of $k_i\times k_i$ matrices  $\wt P_i(J + \mu_i \cdot \Id)$   and  $ c\cdot A$, where  $\wt P_i(t)= \prod_{s\ne i} ( t - \mu_s)^{k_s}$, and
 the matrices $J$ (``Jordan block'') 
  and $A$  (``antidiagonal matrix'')  are   given by   
 $$
J= \begin{pmatrix} 0  & 0      & \cdots & \cdots   & 0\\ 
                   1  & 0      &    0   &  \cdots  & 0\\ 
									 0  & 1      &\ddots  &  \ddots  & \vdots \\
							\vdots 	& \ddots & \ddots &   \ddots & 0  \\
							    0   &  0     & \cdots &    1     & 0   \end{pmatrix} \ \  \textrm{and} \ \  
				A= \begin{pmatrix} 0  & \cdots  & \cdots     &0        & 1\\ 
                           0  & \cdots      &  0     &   1     & 0\\ 
									   \vdots   &   \udots    &\udots  &  0    & \vdots \\
							             0	& \udots & \udots &   \vdots & 0  \\
							             1   &  0     & \cdots &    0     & 0   \end{pmatrix} 					
									.
$$
The $(m+1)$-st block is the product of $d\times d$ matrices $-\wt P_\infty( J)$ and $c \cdot A$,  were $\wt P_\infty(t) = \prod_{s} ({ 1 -t \mu_s})^{k_s}$.
}\end{Remark}

\begin{Remark} \label{rem:complex}{\rm
Some of the  flat  coordinates $Y^k$ can be complex-valued. This may be related to complex roots $\mu_i$, or can appear if  $\mu_i$ is real but  $\det(\mu_i \cdot \Id-L)<0$. In the first  case, the real and the imaginary  parts  of the complex-valued coordinates contribute to real flat coordinates.  
In the second  case,    the  square root  and its derivatives are  pure imaginary at the point $\mu_i$.  The   imaginary  parts of them give real coordinates.  Note that  the matrix of $g^{ij}$  in the flat coordinates $Y^k$ from Remark \ref{rem:matrix} is calculated as the matrix 
of scalar products  of the differentials $\dd Y^i$ and $\dd Y^j$, in particular if $Y^i$ and $Y^j$  are purely imaginary  the  $g$-scalar product of the corresponding differentials is minus the $g$-scalar product of  the differentials of the corresponding imaginary parts.   In other words, when we pass from the   pure imaginary coordinates to their  imaginary parts we should multiply the corresponding blocks  of the matrix  from Remark \ref{rem:matrix} by $-1$. 
}\end{Remark} 

Let us recall two known examples: Example \ref{ex:riem1}  is classical and Example   \ref{ex:riemb} can be found in \cite{blaszak2007,MB}.

\begin{Ex} \label{ex:riem1}{\rm  Assume  the polynomial $P=-4\prod_{i=1}^n (t-\mu_i)$ has  $n$ simple real roots  $\mu_1<...<\mu_n$.  
  Then,   
$$
Y^i = \sqrt{\det(\mu_i \Id - L)}= \sqrt{\prod_{s=1}^n(\mu_i - x^s)}. 
$$
In these coordinates, the matrix of $g$ is diagonal: $g_{ij}= \operatorname{diag}\left(\prod_{s\ne 1}(\mu_1-\mu_s),\dots,\prod_{s\ne n}(\mu_n-\mu_s)\right)$.
After the cosmetic change of the form 
\begin{equation} 
\label{eq:ell} 
Y_{\mathrm{new}}^i= \frac{Y_{\mathrm{old}}^i}{\sqrt{\prod_{s\ne i}(\mu_i-\mu_s)}}  \ \  \  \left( =  \sqrt{\frac{\prod_{s=1}^n (\mu_i - x^s)    }{\prod_{s\ne i}(\mu_i-\mu_s)}}\right) 
\end{equation}  
the metric takes the form $g= \operatorname{diag}(\pm 1,\dots,\pm 1)$.  

In particular, if the metric is Riemannian,    the polynomial $P$ necessarily has $n$  or $n-1$ 
real simple  root and the range of the coordinates  is  as follows: $\mu_1<x^1<\mu_2<\cdots < \mu_n<x^n$ (in the case of $n-1$ simple roots we think that $\mu_1=-\infty$).
In this case, the coordinates  $Y_{\mathrm{new}}^i$   are real-valued and the metric $g$ in these coordinates is the standard Euclidean metric.  Formula \eqref{eq:ell}, known since Jacobi \cite{jacobi} and Neumann  \cite{neumann},   expresses Cartesian coordinates $Y^i$ in terms of ellipsoidal coordinates  $x^i$ and  is widely used in many  physical applications.

}\end{Ex}

\begin{Ex}\label{ex:riemb}{\rm  If $P(t)=t^n$,  the flat coordinates are coefficients of the characteristic polynomial of $L$: 
\begin{equation}
\det(t\cdot \Id - L)= t^{n}+ Y^{n} t^{n-1} + Y^{n-1} t^{n-2}+...+Y^1.
\end{equation} 
In the coordinates $Y^i$, the matrix of $g$ is the antidiagonal matrix $A$  from Remark \ref{rem:matrix}. 
}\end{Ex} 

Next, let us consider the case of metrics $P(L)g_{\mathsf{\mathsf{LC}}}$ of constant nonzero curvatures. In this case, $\deg P(L) = n+1$.

\begin{Theorem} \label{thm:generalised_casimirs}
Assume 
  $P(t)=-4{K}  (t-\mu_1)^{k_1}(t-\mu_2)^{k_2}...(t-\mu_m)^{k_m}$ with $ k_1+...+k_m=n+1$.
Then, the  $n+1$ 
functions 
\begin{equation}  
\label{eq:12} 
Y_i^{\alpha}:= \tfrac{\dd^\alpha}{\dd t^\alpha} \sqrt{\det(t\cdot {\Id}-L)}_{|t=\mu_i}   \mbox{  with $\alpha\in \{0,...,k_{i}-1\}$  \  \textrm{and} \ $i \in \{1,...,m\}$}     
 \end{equation}
form generalised flat coordinates for the (contravariant) metric $g  = P(L) g_{\mathsf{LC}}$. 
\end{Theorem}

\begin{Remark}\label{rem:matrix1}{\rm
Calculations in the proof of Theorem \ref{thm:generalised_casimirs} also give the matrix $G^{ij}$ from Remark \ref{rem:1.2}. 
It is blockdiagonal with $m$  blocks of dimensions $k_1,...,k_m $. The $i$-th  block   is given by the product of the $k_i\times k_i$ matrices   $ \wt P(J + \mu_i \cdot \Id)$   with $ \wt P=     \prod_{s\ne i} (t - \mu_s)^{k_s}$ and  $K\cdot A$, where $J$ and $A$ are as in Remark \ref{rem:matrix}. 
}\end{Remark}

\begin{Remark} \label{rem:1.6}{\rm
Some of the generalised flat coordinates from Theorem \ref{thm:generalised_casimirs}  may be complex-valued; one obtains real-valued generalised flat coordinates as explained in Remark \ref{rem:complex}. 
}\end{Remark}

\subsubsection{(Generalised)  flat coordinates for the metric with  an arbitrary graph $\mathsf F$.}\label{sect:3steps}

We are given an in-direct forest $\mathsf F$ endowed with the following additional data:

\begin{itemize} 
\item the vertices of $\mathsf F$ are labelled by numbers $\alpha = 1,\dots, B$;  
\item each edge $\vec{\beta\alpha}$  of $\mathsf F$ is assigned with a number $\lambda_\beta$;  
\item each vertex $\alpha$ is assigned with a natural number $n_\alpha$ (dimension of the $\alpha$-block),    a diagonal (Nijenhuis, in the terminology of \cite{Nijenhuis}) operator $L_\alpha$  and  a polynomial $P_\alpha$ of degree $\le n_\alpha + 1$.
\end{itemize}

Moreover, depending on the roots $\mu_1, \dots, \mu_m$ of $P_\alpha= c_\alpha \prod_{i=1}^m (t - \mu_i)^{k_i}$ and their multiplicities $k_1,\dots, k_m$,  in the previous Section \ref{sec:flatLC} we have constructed the functions
$$
Y_i^s \quad i = 1,\dots, m, \ \ s = 0,\dots, k_i-1, \quad\mbox{and} 
\quad Y_\infty^s \quad s = 1,\dots, d,
$$
where $d=n_\alpha - \deg P_\alpha$ (the latter functions appears if $\deg P_\alpha < n_\alpha$).
 
According to Theorems \ref{thm:casimirs} and \ref{thm:generalised_casimirs},  all these functions form a complete set of (generalised) flat coordinates for each  {\it one-block} metric $\tilde g_\alpha=P_\alpha(L_\alpha) g_\alpha^{\mathsf{\mathsf{LC}}}$.   Notice that the number of such functions is $n_\alpha$ if $\tilde g_\alpha$ is flat, and is $n_\alpha + 1$ if $\tilde g_\alpha$ has constant non-zero curvature $K_\alpha$.

For each $\alpha$-block the above collection of (generalised) flat coordinates is characterised by a constant matrix  $G_\alpha$ (see Remark \ref{rem:1.2}), which is explicitly described in Remarks \ref{rem:matrix} and \ref{rem:matrix1} for the cases $K_\alpha = 0$ and $K_\alpha\ne 0$ respectively.

Besides these functions, for each flat block, we also introduce an auxiliary function
\begin{equation}  \label{eq:bols_8}
\phi := \tfrac{1}{c} \left(\tfrac{1}{(d+1)!}\tfrac{\dd^{d+1}}{\dd t^{d+1}} \sqrt{\det(\Id - t\cdot {L})}_{|t=0} \   - \,\tfrac{1}{2}R^{(d+1)}\right),  
\end{equation}
where $R^{(d+1)}$ is defined by \eqref{eq:bols_5}.  The formula reminds the definition of $Y_\infty^s$ for $s=d+1$ (cf. \eqref{eq:11}).  However,  $\phi$ is not a flat coordinate but satisfies the following important relations (the proof will be given in Section \ref{sect:proof1.5-1.6}):
$$
g^* (\dd Y^i , \dd \phi) =  Y^i \quad\mbox{and} \quad g^*(\dd \phi, \dd \phi) =   d \phi,
$$
where $g = \tilde g_\alpha$ is a one-block metric.

Our goal is to construct a complete set of (generalised) flat coordinates for the {\it multi-block} metric 
$$
g = \operatorname{diag}(g_1,\dots,g_B) =  \operatorname{diag}(f_1 \tilde g_1,\dots, f_B\tilde g_B),
$$ 
defined by \eqref{eq:hatg}
from the above described coordinates on individual blocks and also to describe the corresponding matrix $G$ via the matrices $G_1, \dots, G_B$. 

The construction consists of three steps.  

{\bf Step 1.}    Removing ``redundant'' coordinates from each block.

For each $\alpha$-block, consider all the incoming edges $\vec{\gamma_i \alpha}$ related to the $\gamma_i$-blocks of non-zero curvature $K_{\gamma_i}\ne 0$. According to our construction, these edges are endowed with numbers $\lambda_{\gamma_i}$'s each of which is a simple root of the polynomial $P_\alpha$ so that the collection of (generalised) flat coordinates for the $\alpha$-block contains the functions
$$
Y_{\gamma_1}^0 = \sqrt{\det(\gamma_i\cdot \Id - L_\alpha)},  \  \dots \  , Y_{\gamma_s}^0 = \sqrt{\det(\gamma_s\cdot \Id - L_\alpha)}.
$$
Notice that the corresponding matrix $G_\alpha$ has the following form (cf. Proposition \ref{prop:bols1} below)
$$
G_\alpha = \operatorname{diag}(c_1, \dots, c_s, \widetilde G_\alpha) 
$$
where $\widetilde G_\alpha$ is the matrix for the remaining flat coordinates and $c_i=(G_\alpha)^{ii}$ is defined from
the relation $g_\alpha^* (\dd Y_{\gamma_i}^0, \dd Y_{\gamma_i}^0) = c_i - K_\alpha (Y_{\gamma_i}^0)^2$.

We remove the functions from the set of (generalised) flat coordinates on the $\alpha$-block and reduce the matrix $G_\alpha$ to $\widetilde G_\alpha$ respectively. 

As we see,  each incoming edge related to a neighbouring {\it block of non-zero curvature} leads to removing one of the previously constructed coordinates. Each incoming edges related to a {\it flat block} will result in a certain modification of  one of these coordinates.      

{\bf Step 2.}   Modifying some of coordinates on each block.

For each $\alpha$-block consider all the incoming edges $\vec{\beta_j \alpha}$ related to the neighbouring flat $\beta_j$-blocks. According to our construction, the corresponding numbers $\lambda_{\beta_j}$'s related to these edges are roots of $P_\alpha$ of multiplicity at least two.   Some of these numbers may coincide so that the $\beta_j$ blocks can be naturally partitioned into groups related to equal labels.  The modification procedure will be applied separately to each of these groups.  W.l.o.g. assume that $\lambda_{\beta_1} = \lambda_{\beta_2} = \dots = \lambda_{\beta_p}=\lambda$.  Then our list of (generalised) flat  coordinates  on the $\alpha$-block contains the functions
$$
Y_\lambda^i = \tfrac{1}{i!} \tfrac{\dd^i}{\dd t^i}{}_{|\, t=\lambda}\sqrt{\det(t\cdot \Id-L_\alpha)}, \quad i=0,\dots, k-1,
$$ 
where $k$ denotes the multiplicity of $\lambda$ (as a root of $P_\alpha$).  Recall that for each flat $\beta_j$-block we have (uniquely) defined a function $\phi_{\beta_j}$.  We now modify the function $Y_\lambda^{k-1}$ as follows (cf. Proposition \ref{prop:bols2} below): 
$$
Y_\lambda^{k-1}  \mapsto \widetilde Y_\lambda^{k-1} = Y_\lambda^{k-1} - b\, Y_\lambda^0 \cdot \sum_{j=1}^p \phi_{\beta_j}
$$
where $b=\tilde g_\alpha^* (\dd Y_\lambda^0, \dd Y_\lambda^{k-1})$   (notice that $b\ne 0$!). 

To summarise, each incoming edge $\vec{\gamma \alpha}$ to the vertex $\alpha$ requires a certain modification of the (canonically chosen) set of (generalised) flat coordinates on the $\alpha$-block.  This modification depends on whether the corresponding neighbouring $\gamma$-block is flat or not.  Now each $\alpha$-block possesses a new (reduced and modified) set of functions which we denote by $u_{\alpha}^1, \dots , u_{\alpha}^{\tilde n_\alpha}$ where 
$$
\tilde n_{\alpha} = 
\begin{cases}
n_\alpha - s, & \mbox{if the $\alpha$-block is flat,} \\
n_\alpha +1 - s, & \mbox{if the $\alpha$-block has non-zero constant curvature,} \\
\end{cases} 
$$   
where $s$ is the number of incoming edges $\vec{\gamma \alpha}$ related to $\gamma$-blocks of non-zero curvature.
   
{\bf Step 3.}  To get a complete set of (generalised) flat coordinates for the metric 
$$
g = \operatorname{diag} ( g_1,\dots, g_B),    \qquad  g_\alpha  = f_\alpha \tilde g_\alpha =
f_\alpha P_\alpha(L_\alpha) g_\alpha^{\mathsf{LC}},
$$
we now replace the functions $u_{\alpha}^i$ by the following simple rule: 
$$
u_\alpha^ i  \mapsto  \tilde u_\alpha^i =   (f_\alpha)^{-1/2} \cdot u_\alpha^i ,   
$$
where $f_\alpha$ is defined by \eqref{eq:falpha} so that, in more detail,  the ``new''  generalised flat coordinate is
\begin{equation}
\label{eq:bls1} 
\tilde u_\alpha^i  = u_\alpha^i   \cdot \!\!\! \prod_{\begin{array}{c}s\stackrel{\prec}{=} \alpha \\
 \textrm{$s$ not a root} \end{array}} \!\!\!\!\!\! \sqrt{\det(L_{\operatorname{next}(s)}-\lambda_{s}\cdot\Id_{n_s})}.
 \end{equation}
 If $\alpha$ is a root, then $f_\alpha = 1$ and $\tilde u_\alpha^i  = u_\alpha^i$.

The final conclusion is as follows.
\begin{Theorem}\label{thm:F_casimirs}
The above constructed functions $\tilde u_\alpha^i$ form a complete set of independent (generalised) flat coordinates for $g = \operatorname{diag} \left(  g_1,\dots,  g_B\right)$ defined by \eqref{eq:hatg}.  Moreover, the corresponding (constant) matrix $G$ is block-diagonal and has the form
$$
G=\operatorname{diag}\left(\widetilde G_1 , \dots, \widetilde G_B\right)
$$
where $\widetilde G_\alpha$ is the reduced matrix\footnote{Notice that in general the size $\tilde n_\alpha$ of $\widetilde G_\alpha$ does not coincide with the dimension $n_\alpha$ of the $\alpha$-block.} for the $\alpha$-block obtained by removing ``redundant'' coordinates (see Step 1).
\end{Theorem}

\weg{

 \subsection{Special properties of the Riemannian case} 

Let us first consider the case of one block.  Assume the metric $P(L)g_{\mathsf{LC}}$  is definite. 
This  implies that all polynomials $P_\alpha$ have degree at least $n_\alpha$ and has simple real roots. 
We denote them by  $\mu_1 <\mu_2<...$. If the metric is positive definite and the curvature is positive or zero, 
the condition  that all its components $P(x^i) \prod_{s\ne i} (x^i -x^s)$ is positive implies then that 
$\mu_i < x^i < \mu_{i+1}$ for every $i\in\{1,...,n\}$; moreover, if the curvature is positive then $x^n<\mu_n+1$.

		In the case of negative curvature, 	  the situation is more complicates. In the case, there are $n+1$ roots covered with their multiplicity, 
		which we denote by $\mu_1\le \mu_2\le ... \mu_{n+1}$, from which  at least $n-1$ are simple so at most one ``$\le$'' in the last inequality is not ``$<$''.
		Then, assuming the metric is positive definite,  and using that the diagonal components of the metric are positive, one sees that 
		there exists $j \in \{1,...,n +1 \}$   such that 			all $\mu_{i-1}<x_\alpha^i<\mu_{i}$ for $i<j  $  and $\mu_{i+1}<x_\alpha^i\mu_{i+2}$ for $i \ge j$. In these  inequalities we think that $\mu_0= -\infty$ and $\mu_{n +2}= +\infty$.  }

\section{Proof of Theorem  \ref{thm:main}}  \label{sec:proof:main}

 \subsection{Equations  corresponding to separating  coordinates} \label{sec:2.1}

 We assume that  
\begin{equation} \label{eq:diagonalmetric} g= \sum_i \varepsilon_i \exp(g_i) \  (\dd x^i)^2,\end{equation}
where $g_i$ are functions of $x^1,...,x^n$ and
 $\varepsilon_i \in \{1, -1\}$.    
Our goal is to find all functions $g_i$ such that the metric $g$ has constant curvature and $(x^1,...,x^n)$  are separating coordinates, and to show that they (up to the freedom described in Theorem  \ref{thm:main}) are as in \eqref{eq:hatg}. Note that the first statement of  Theorem  \ref{thm:main}, namely that 
\eqref{eq:hatg}  has constant curvature, follows from the existence of (generalised) flat coordinates, see Remark \ref{rem:new}. 

Since the  metric has constant curvature,  for every $i\ne j \ne k\ne i$  we have 
$R^k_{\ jik}=0$, which is equivalent to
\begin{equation} \label{eq:riem1231}{\frac{\partial^{2}}{\partial x^{j}\partial x^{i}}{g_k}}{}+
\tfrac{1}{2} \left( \left(\frac{\partial}{\partial x^{i}}{g_k}\right)\left(\frac{\partial}{\partial x^{j}}{g_k} \right) -
\left(\frac{\partial}{\partial x^{i}}{g_j} \right) \left(\frac{\partial}{\partial x^{j}}{g_k} \right)-{\left(\frac{\partial}{\partial x^{i}}{g_k} \right) \left(\frac{\partial}{\partial x^{j}}{g_i} \right)}\right)=0. 
\end{equation}

Next, we will use the condition that  $x^1,...,x^n$ are separating coordinates. Recall that 
 by \cite{LC1904} (more recent references are e.g. \cite{eisenhart,Kress-book} or \cite[Lemma 2.2]{benenti1}), the coordinates $x^1,...,x^n$ are separating for $g$ given by \eqref{eq:diagonalmetric}
 if and  only if the `Hamiltonian' function $H=\tfrac{1}{2} \sum_{i} \varepsilon_i\exp(-g_i) p_i^2$ satisfies    
\begin{equation}\label{eq:LC}
\frac{\partial H}{\partial p_i} \frac{\partial H}{\partial p_j} \frac{\partial^2 H}{\partial x^i\partial x^j} - 
\frac{\partial H}{\partial x^ i} \frac{\partial H}{\partial p_j} \frac{\partial^2 H}{\partial p_i\partial x^j} -
\frac{\partial H}{\partial x^j} \frac{\partial H}{\partial p_i} \frac{\partial^2 H}{\partial p_j\partial x^i}=0\quad
\mbox{for all $i\ne j$}.
\end{equation} 
Condition \eqref{eq:LC} is polynomial in $p$ of   degree  $4$.    Vanishing of the coefficient   at 
$p_k^2 p_i p_j$ with   $i\ne j \ne k \ne i$,   gives the following equation (known in the literature, see e.g. \cite[Lemma 2.2]{benenti1} or \cite[Corollary 1]{Kalnins-book}):
\begin{equation}  \label{eq:1}
\frac{\partial^{2}}{\partial x^{i}\partial x^{j}}{g_k}- \left( \left(\frac{\partial}{\partial x^{i}}{g_k}\right)\left(\frac{\partial}{\partial x^{j}}{g_k} \right) -
\left(\frac{\partial}{\partial x^{i}}{g_j} \right) \left(\frac{\partial}{\partial x^{j}}{g_k} \right)-{\left(\frac{\partial}{\partial x^{i}}{g_k} \right) \left(\frac{\partial}{\partial x^{j}}{g_i} \right)}\right)=0.
\end{equation}

Comparing  \eqref{eq:riem1231} with \eqref{eq:1}, we see that for $i\ne j \ne k \ne i$ the following two conditions should hold:
\begin{eqnarray}
0 &=& \frac{\partial^{2}}{\partial x^{i}\partial x^{j}}{g_k} \label{eq:2}  \\
0 & =&  \left(\frac{\partial}{\partial x^{j}}{g_i} \right)
 \left(\frac{\partial}{\partial x^i}{g_k} \right)+\left(\frac{\partial}{\partial x^{j}}{g_k} \right) \left(\frac{\partial}{\partial x^{i}} {g_j}\right)  -\left(\frac{\partial}{\partial x^{j}}{g_k} \right) \left(\frac{\partial}{\partial x^{i}}{g_k} \right). \label{eq:3}  
\end{eqnarray} 
 
Next, vanishing of   the coefficient of \eqref{eq:LC} at  $p_j^3 p_i$ (with $i\ne j$) gives
\begin{equation}\label{eq:4}
\frac{\partial^{2}}{\partial x^{i}\partial x^{j}}{g_j}=- \left(\frac{\partial}{\partial x^{j}}{g_i} \right) \left(\frac{\partial}{\partial x^{i}}{g_j} \right).
\end{equation}

In our next step, we use  equation  (\ref{eq:4}) to construct one more second order equation on the functions $g_i$.
 Take   $i\ne j$ and  assume first 
$\frac{\partial}{\partial x^{i}}{g_j} \ne 0$.  Then,  
make the (nonrestrictive)  ansatz  
\begin{equation} 
\label{eq:ansatz}  \frac{\partial^{2}}{{\partial x^{i}}^2}{g_j} = 
-\left(\frac{\partial}{\partial x^{i}}{g_j} \right)^{2}+u_{ij}(x) \left(\frac{\partial}{\partial x^{i}}{g_j} \right)
\end{equation} 
for a certain function  $u_{ij}(x)$.  
Next, differentiate  \eqref{eq:ansatz}
with respect to $x^j$ and subtract   the derivative of \eqref{eq:4} with respect to $x^i$. The left 
hand
side becomes zero,  and after substituting the second derivatives of $g_i$ and $g_j$ given by  \eqref{eq:4} and \eqref{eq:ansatz} into the right hand side, we get 
$$
0=-\left(\frac{\partial}{\partial x^{j}}u_{ij} \left(x\right) \right) \left(\frac{\partial}{\partial x^{i}}\mathit{g_j} \right). 
$$ 
By our assumption $\frac{\partial}{\partial x^{i}}{g_j} \ne 0$, so the function $u_{ij}$ does not depend on the coordinate  $x^j$. 

Next, consider the case $\frac{\partial}{\partial x^{i}}{g_j}\equiv  0$. In this case, \eqref{eq:ansatz} is trivially fulfilled with an arbitrary function $u_{ij}$.   Thus,  we may and will assume that  the function $u_{ij}$   depends   on   $x^i$ only. 
That is,   
\begin{equation} 
\label{eq:5bis}  
\frac{\partial^{2}}{\partial {x^{i}}^{2}}{g_j} = 
-\left(\frac{\partial}{\partial x^{i}}{g_j} \right)^{2}+u_{ij}(x^i) \left(\frac{\partial}{\partial x^{i}}{g_j} \right)
\end{equation} 
for  certain functions  $u_{ij}(x^i)$.

 \subsection{Partial nonstrict order on the set of indices} \label{sec:order}  

We consider the metric \eqref{eq:diagonalmetric} such that the diagonal components are analytic 
 and   \eqref{eq:2}, \eqref{eq:3},  \eqref{eq:4}  are fulfilled.  Recall that they are fulfilled if the metric has constant curvature and the coordinates are separating. Next, by \cite{kruglikov} the Killing tensors for real-analytic metrics are real-analytic which implies that we may assume without loss of generality that the separating coordinates are also analytic.

 Let us define a relation $\preccurlyeq$ on the set $\{1,...,n\}$ of indices:  
$$
i \preccurlyeq j  \  \textrm{if and only if }  i=j \ \textrm{or} \ g_j  \ \textrm{depends on $x^i$} . 
$$

We assume that the functions $g_i$ are analytic in  the coordinates $x^j$, so the notion  {\it 
 depends on $x^i$} is well defined.

\begin{Lemma} The relation ``$\preccurlyeq$''
is a nonstrict partial order, that is,
\begin{equation}\label{def:order} 
\textrm{  if $i \preccurlyeq j $ and $j \preccurlyeq k$ then $i\preccurlyeq k$.} \end{equation} 
Moreover, for every $k$, the relation  ``$\preccurlyeq$'' restricted to  the set $\operatorname{Less}_{k}  := \{ i \mid i \preccurlyeq k\}$ is a 
nonstrict order, that is, for every $i,j\in \operatorname{Less}_k$ we have  $j\preccurlyeq i$ or $i \preccurlyeq   j$ (or both). 
\end{Lemma}

\begin{proof} If two of the three indices $i,j,k$ are equal, there is nothing to prove. We assume that they are all different so that \eqref{eq:3} holds. 
By the definition of $\preccurlyeq$,  we have 
$\tfrac{\partial g_j}{\partial x^i}\ne 0\, , \tfrac{\partial g_k}{\partial x^j}\ne 0$. 
Assuming, by contradiction, that  
$\tfrac{\partial g_{k}}{\partial x^i}=0$ and plugging it in  
\eqref{eq:3}, we obtain 
$$
0  =  \left(\frac{\partial}{\partial x^{j}}{g_i} \right)
 \underbrace{\left(\frac{\partial}{\partial x^i}{g_k} \right)}_{=0}+\underbrace{\left(\frac{\partial}{\partial x^{j}}{g_k} \right)}_{\ne 0} \underbrace{\left(\frac{\partial}{\partial x^{i}} {g_j}\right)}_{\ne 0}  -\left(\frac{\partial}{\partial x^{j}}{g_k} \right) \underbrace{\left(\frac{\partial}{\partial x^{i}}{g_k} \right)}_{=0},
$$ 
which is impossible.  Hence,  $\tfrac{\partial g_{k}}{\partial x^i}\ne 0$, i.e. $i\preccurlyeq k$, as stated. 
Similarly, assuming  $i, j\in \operatorname{Less}_k$ we see that  $i \not\preccurlyeq j$ and $ j\not\preccurlyeq  i$ lead to a contradiction in view of     
$$
0  = \underbrace{ \left(\frac{\partial}{\partial x^{j}}{g_i} \right)}_{=0}
 \left(\frac{\partial}{\partial x^i}{g_k} \right)+\left(\frac{\partial}{\partial x^{j}}{g_k} \right) \underbrace{\left(\frac{\partial}{\partial x^{i}} {g_j}\right)}_{= 0}  -\underbrace{\left(\frac{\partial}{\partial x^{j}}{g_k} \right)}_{\ne 0} \underbrace{\left(\frac{\partial}{\partial x^{i}}{g_k} \right)}_{\ne 0}.
$$
\end{proof}

Next, with the help of relation $\preccurlyeq$ we define the relation $\approx$ as follows:
\begin{equation}\label{def:approx}
i\approx j   \ \Longleftrightarrow \ \bigl( \ i\preccurlyeq j \ \textrm{and}   \ j\preccurlyeq i \bigr).
\end{equation}
It is clearly an equivalence relation. Reflexivity $i \approx i$  is fulfilled because $i\preccurlyeq i$ by the definition of $\preccurlyeq$. Symmetry is clear, since the definition \eqref{def:approx} is symmetric with respect to $i\longleftrightarrow j$. Transitivity follows from  \eqref{def:order}.   

We now consider equivalence classes with respect to the equivalence relation $\approx$. 
The nonstrict partial order $\preccurlyeq$ on $\{1,...,n\}$ defines a strict partial order $\prec$ on the set of equivalence classes by the standard procedure:
$$
[i] \prec [j] \ \Longleftrightarrow \ \left( \  i \preccurlyeq j \ \textrm{and } \ [i] \ne  [j] \ \right).
$$
The  partial order is well defined and in particular does not depend on the choice of the elements 
 of the equivalence classes $[i]$ and $[j]$. Moreover, for every $[k]$ the set 
$\operatorname{Less}_{[k]}:=   \{ [i] \mid [i] \prec [k]\}$ is totally ordered.

Let $B$ be  the number of equivalence classes. We denote different equivalent classes by different  numbers 
  $\alpha \in \{1,...,B\}$.  The order $\prec$ gives us a partial  order on the set $\{1,...,B\}$. 

We will think without loss of generality that the coordinates are numerated in such a way that the first $n_1$ indices 
form the first equivalence class, the next $n_2$ indices form the second and so on. In other words, we assume that the coordinate system 
is 
$$
(x^1,...,x^n)= \left(\underbrace{x^1_1,...x_1^{n_1}}_{X_1},...,\underbrace{x^1_\alpha,...x_\alpha^{n_\alpha}}_{X_\alpha},..., \underbrace{x^1_B,...x_B^{n_B}}_{X_B}\right). 
$$
The blocks $X_\alpha$'s  correspond to equivalence classes of indices, 
so that for every $\alpha \in \{1,..., B\}$ and 
 two indices $i\ne  j$ with $i,j \le n_\alpha$  we have 
$\tfrac{\partial g^\alpha_i}{\partial x_\alpha^j} \ne 0 $ and $\tfrac{\partial g^\alpha_j}{\partial x_\alpha^i} \ne 0 $. For every $\alpha\ne \beta$ and two 
indices $i\le n_\alpha$ and $j\le  n_\beta$ we have
$\tfrac{\partial g^\beta_j}{\partial x_\alpha^i} \ne 0 $ if and only if $\alpha \prec \beta$.  Here by $g^\alpha_j$ we  denote the function 
$g_j$ corresponding to the coordinate $x_\alpha^j$.

 	Next, 
	let us recall that partial orders such that for every $\alpha$ the subset $\operatorname{Less}_\alpha$ is totally ordered  are closely related to in-directed rooted forests. 
	More precisely, for such a partial order on a finite set $\{1, ..., B\}$ one can canonically 
	construct   an in-directed rooted forest whose vertices are elements $\alpha\in \{1, ..., B\}$. Namely, we connect two vertices $ \alpha \ne \beta$ by  an oriented edge $\vec {\beta \alpha} $ if and only if $\alpha \prec \beta$ and  there is no $\gamma$ such that $\alpha\prec\gamma  \prec \beta$. Conversely, an in-directed rooted forest defines a natural partial order on the set of vertices:
  vertices   $\alpha$ and $\beta$  satisfy $\alpha \prec \beta$ if and only if there exists an oriented  
way from $\beta$ to $\alpha$. These two constructions are mutually inverse.
 
	\subsection{Proof of Theorem \ref{thm:main} for flat metrics 
	under the  additional  assumption that  $\mathsf F$  is a chain}
	Recall that  for an 
 in-directed forest $\mathsf F$ 	 a {\it chain} is an oriented way from a leaf to a root. Equivalently, in terms of the strict partial order $\prec$, {\it chain} is a maximal totally ordered subset.  We say that  $\mathsf F$  is a {\it chain}, if it contains only one leaf and therefore only one chain. Therefore, 
 the relation  $\preccurlyeq$ is defined for any two indices $i\ne j$ (i.e., we always have $i\preccurlyeq j$ or/and  $j \preccurlyeq i$. 
In other words, for any $i\ne j$ we have $\frac{\partial}{\partial x^{j}}{g_i} \ne 0$, or $\frac{\partial}{\partial x^{i}}{g_j} \ne 0$, or both.

In this section we  assume that the metric $g$ is flat, has the form \eqref{eq:diagonalmetric} in separating coordinates $x^1,...,x^n$, and the in-directed forest $\mathsf F$  constructed by the metric in \S \ref{sec:order} is a  chain.   
Without loss of generality we think that the 
coordinates $(x^1,...,x^n)$ are organized into blocks  
$$
(x^1,...,x^n)= \left(\underbrace{x^1_1,...x_1^{n_1}}_{X_1},...,\underbrace{x^1_\alpha,...x_\alpha^{n_\alpha}}_{X_\alpha},..., \underbrace{x^1_B,...x_B^{n_B}}_{X_B}\right). 
$$
such that the order $\prec$ on $\{1,...,B\}$ is induced by  the canonical order $<$ on the set of natural numbers.

Let us now show that in this setup the functions  $u_{ij}(x^i)$ from 
\eqref{eq:5bis} are very special: namely $u_{ij}=u_{ik}$  for every $j\ne k$ (we assume that $\frac{\partial}{\partial x^{i}}{g_j} \ne 0$  and $\frac{\partial}{\partial x^{i}}{g_k} \ne 0$ because otherwise the corresponding 
function does not come into the equation \eqref{eq:5bis}).

To that end, we differentiate \eqref{eq:3}  with respect to $x^i$ and substitute the second derivatives of $g_{j}$ and $g_k$ given by (\ref{eq:2}, \ref{eq:5bis}). The result can be  simplified using 
\eqref{eq:3}  and we obtain 
$$ 
\left(\frac{\partial}{\partial x^{i}}{g_j} \right) \left(\frac{\partial}{\partial x^{j}}{g_k} \right) (u_{ij}- u_{ik}) =0.
$$ 
 Since  $\left(\frac{\partial}{\partial x^{i}}{g_j} \right)\ne 0$ and $ \left(\frac{\partial}{\partial x^{j}}{g_k} \right)\ne 0$,
 we obtain $u_{ij}=u_{ik}$, as claimed. Similarly, by replacing $j \longleftrightarrow k$,  we obtain that  $ \left(\frac{\partial}{\partial x^{j}}{g_k} \right)\ne 0$ implies $u_{ik}= u_{ij}$. 
It remains to note that by our assumptions, since 
$\mathsf F$   is a chain, one of the conditions $ \left(\frac{\partial}{\partial x^{j}}{g_k} \right)\ne 0$ and $ \left(\frac{\partial}{\partial x^{k}}{g_j} \right)\ne 0$ must hold and we are done. 

Finally, in our setup, the following equation should be fulfilled  for all $i\ne j   \ne  k\ne i$ and for certain functions $u_i(x^i)$ 
 (denoted previously by $u_{ij}(x^i)$):
 \begin{equation} \label{eq:5}   \frac{\partial^{2}}{\partial {x^{i}}^2}{g_j} = 
-\left(\frac{\partial}{\partial x^{i}}{g_j} \right)^{2}+u_{i}(x^i) \left(\frac{\partial}{\partial x^{i}}{g_j} \right).\end{equation}

Next, observe that  the equations  \eqref{eq:2}, \eqref{eq:3},  \eqref{eq:4}, \eqref{eq:5} were  the starting point of the proof of   
\cite[Theorems 4 and 5]{NijenhuisAppl3}. The  equations in \cite{NijenhuisAppl3}  came from other assumptions than those in the present paper, but still they are precisely the same equations. 
Namely,  \eqref{eq:2} is \cite[(51)]{NijenhuisAppl3}, \eqref{eq:3} is \cite[(64)]{NijenhuisAppl3}, \eqref{eq:4}
 is \cite[(54)]{NijenhuisAppl3}, \eqref{eq:5} is \cite[(57)]{NijenhuisAppl3}. 
In \cite{NijenhuisAppl3} it was additionally  assumed that $g$ is flat, i.e., has zero sectional curvature.

Under the following assumptions,  
 it was shown that  $g$ 
is given by \eqref{eq:hatg} as in Theorem \ref{thm:main}:
\begin{enumerate}
\item The metric $g$ is diagonal as in \eqref{eq:diagonalmetric}  and has zero sectional curvature.
\item Equations \eqref{eq:2}, \eqref{eq:3},  \eqref{eq:4}, \eqref{eq:5} are fulfilled.
\end{enumerate}

    Let us point out and explain the main steps of the proof from   
\cite{NijenhuisAppl3}. Using solely the equations (\ref{eq:2},\ref{eq:4},\ref{eq:5}), in \cite[\S 6.2]{NijenhuisAppl3} it was shown that the metric is given by \cite[(62)]{NijenhuisAppl3} which is 
$$
g_{ii}= \varepsilon_i e^{g_i} = h_i(x^i) \left(\prod_{s\ne i} \left( C_{is} U_i(x^i)+ C_{si} U_s(x^s)+ E_{is}\right)^{\alpha_{is}}  \right)
$$
for certain functions $U_i(x^i)$, $h_i(x^i)$ 
 and constants $\alpha_{ij}\in \{0,1\}$, $C_{ij}$ and $ E_{ij}$. 

Next, heavily using \eqref{eq:3},  it was shown in \cite[\S6.3 and \S 7.1]{NijenhuisAppl3}  that by a change of variables one can achieve 
$E_{ij}=0$ and make 
 the matrix $C_{ij}$  such that it comes from an in-directed  forest by a procedure described in  \cite[\S 4.1 and \S 7.1]{NijenhuisAppl3}
(in \cite[\S 4.1]{NijenhuisAppl3} it is explained how to construct the matrix $c_{ij}$ from the in-directed forest and in \cite[\S 7.1]{NijenhuisAppl3} it is explained how the matrix $C_{ij}$ is related to $c_{ij}$).   
Finally, in   \cite[end of \S 7.1 and \S 7.2]{NijenhuisAppl3},  it is explained that the condition that the curvature is zero implies that the ``blocks''  $g_\alpha$  of the metric  
and the ``warping coefficients'' 
$\sigma_\alpha$  are as  Theorem \ref{thm:main}, which completes  the proof of Theorem \ref{thm:main} under the additional assumption that the curvature of $g$  is zero and that the corresponding in-directed forest is a chain.

	\subsection{Proof of Theorem \ref{thm:main} for flat metrics}

We take our  in-directed forest $\mathsf F$ constructed as in \S \ref{sec:order} and choose a chain in  it.  Let this  chain 
have $B'$   vertices   and assume,
without loss of generality,  that they correspond to the first $B'$ 
blocks, that is, $1 \prec 2 \prec \cdots \prec B'$ and for any $\alpha\in \{1,...,B'\}$ and 
  $\beta \in \{B'+1,B'+2,...,B\}$ we have $ \beta \not \prec \alpha $.

	 By construction, the components 
	$g_1,g_2,....,g_{n_1+n_2+...+n_{B'}}$ do not  depend on the coordinates from the blocks $X_{B'+1},...,X_B$. 
	
	Next, consider the $n_1+n_2+...+n_{B'}$-affine subspace corresponding to the first $n_1+n_2+...+n_{B'}$ coordinates  (i.e., set the remaining coordinates to be constant). The restriction of the metric $g$ to this affine subspace  is diagonal  and its diagonal components are the first 
	$n_1+n_2+...+n_{B'}$ components of $g$. Since they do not depend on the remaining coordinates, the affine subspace is a totally geodesic submanifold.  Then, it has constant curvature. Now, for a Killing tensor its    restriction   to a totally geodesic submanifold is  a Killing tensor. If the Killing tensor   is diagonal in coordinates $x^1,...,x^n$ (with different eigenvalues of $K^i_j$), then the restriction is also diagonal with different eigenvalues. Then,   the restriction of the metric satisfies  the assumptions of   the previous section, in particular the restriction of the metric to the affine subspace  is constructed by the data (chain, admissible labels) as described in \S \ref{sec:2}. Then, the metric $g$ is  blockdiagonal  with blocks 
	of the form $f_\alpha P_\alpha(L) g_\alpha^{\LC}$ and the coefficients of $P_\alpha$ satisfy the conditions (ii), (iii) of \S \ref{sec:2}.

If the graph  $\mathsf F$  consists of more than one connected component, then by the construction the metric is the direct product of the metrics corresponding to the components. This implies that the curvature of the metric must be zero implying the condition (i) of \S \ref{sec:2}.  Theorem \ref{thm:main} is proved for flat metrics. 

The case of constant nonzero curvature will be reduced to the case of curvature zero in the next section.

\subsection{Proof of Theorem \ref{thm:main} for metrics of constant  nonzero curvature} 

Assume that $g$ has constant nonzero curvature $K$ and  $(x^1,...,x^n)$ are separating coordinates for it, in particular $g=\operatorname{diag}(g_{11},...,g_{nn})$.  Our goal is to show that $g$ corresponds to the statement  of Theorem \ref{thm:main}.
Without loss of generality we may and will assume that $K=1$,  this can always be achieved by multiplying 
the metric by a  constant.  

Our plan is to reduce  the problem to the already solved flat case. In order to do it, consider the 
cone over our manifolds;  that is, consider  $ \wt M=  \mathbb{R}_{>0} \times  M^n$  (the coordinate on $\mathbb{R}_{>0}$ will be denoted by $x^0$)  
equipped with the metric 
\begin{equation} \label{eq:hatG}  
\wt g= (\dd x^0)^2 + (x^0)^2 g.
\end{equation} 

It is well known, see e.g. \cite{fedorova, monoud},   that the metric is flat and that 
 the  Levi-Civita connection $\wt\nabla=\{\wt\Gamma^{\ti}_{\tj \tk}\}$
  corresponding to metric $\wt g$ on $\wt M$ is given  by:
  \begin{gather}
    \wt\Gamma^0_{0 0}=0, \ 
    \wt\Gamma^i_{0 0}=0,\\
    \wt\Gamma^0_{j 0}=\wt\Gamma^0_{0 k}=0,\\
    \wt\Gamma^i_{j 0}=\frac{1}{x^0}\,\delta^i_j,\qquad \wt\Gamma^i_{0
      k}=\frac{1}{x^0}\,\delta^i_k,\\
    \wt\Gamma^0_{j k}={}-x^0\cdot g_{j k}(x),\\
    \wt\Gamma^i_{j k}=\Gamma^i_{j k}(x),
  \end{gather}
  where $\Gamma^i_{j k}$ are Christoffel symbols of the Levi-Civita
  connection of $g$, 
	and indices $i,j,k$ ran  from $1$ to $n$.

 Let $K_{ij}$ be a Killing tensor for $g$,  which is diagonal in the coordinates $x^1,...,x^n$ and
such that the eigenvalues of $K^i_j$ are all different; without loss of generality we assume that non of the eigenvalues of $K^i_j$ is zero  since one can achieve it by adding $\const \cdot g$ to $K_{ij}$.  
We consider the  $(0,2)$ tensor field  on $\wt M$   given by   $\wt K= \operatorname{diag}(0,K_{11},...,K_{nn})$. 
By direct calculations using the above formulas for    $\wt\Gamma $,
we see that    $\wt K$ is a Killing tensor for $\wt g$. Its eigenvalues are clearly different. Therefore, the coordinates $x^0,...,x^n$ are separating for $\wt g$.

Next, by \eqref{eq:hatG},   
 the  components $\wt g_{11}$,..., $\wt g_{nn}$ depend on the variable $x^0$, and the component $\wt g_{00} $ does not depend on the coordinates $x^1,...,x^n$.
This implies   that 
  the graph of the separating coordinates $x^0,...,x^n$ has only one root (so it is connected), the dimension $n_\alpha$ corresponding to the root is equal to one which implies that the root has   degree one. 
The corresponding metric is  $g_\alpha=(\dd x^0)^2$ and the only component of the corresponding $(1,1)$  tensor $L$ is $(x^0)^2$. 

 By taking out  the root and the incident edge, we obtain 
a labelled  in-directed tree whose labels satisfy conditions {\rm (ii), (iii)} of \S \ref{sec:2}. If we denote by $\tilde g$ the metric constructed by this labelled tree, then 
$\wt g= (\dd x^0)^2 + (x^0)^2 \tilde g$. Comparing this with \eqref{eq:hatG}, we see that $  g= \tilde g$ impliying that 
the metric $  g$ is constructed by this labelled in-directed tree as we claimed. Theorem \ref{thm:main} is proved.

\section{Proof of Theorems  \ref{thm:2} and \ref{thm:3} }

The proofs are  organised as follows. We first recall/explain how the metrics $g_\alpha$ given by \eqref{eq:tildeg}   and $g$ given by \eqref{eq:hatg}  are related to geodesically equivalent metrics and warp product decompositions. 
Both relations will be used in the proof. The main result of \S \ref{sec:1.3.1} is Theorem \ref{thm:L}. It is interesting on its own and provides  additional  new tools for the theory of separation of variables which  will be used in our  proofs  of Theorems \ref{thm:2} and \ref{thm:3}.

\subsection{Geodesically equivalent metrics, their relation to separation of variables, and the essential uniqueness of $L$}  \label{sec:1.3.1}

Two metrics $g$ and $ \bar g$ on $M^n $  are {\it geodesically equivalent}  if their geodesics, viewed as unparameterised curves,  coincide. Geodesically equivalent metrics is a classical topic in differential geometry and first nontrivial results were obtained by already E. Beltrami \cite{beltrami}, U. Dini \cite{dini} and T. 
Levi-Civita \cite{LC1896}. We will use  certain facts from  the theory of geodesically equivalent metrics below. Let us recall the relation between geodesically equivalent metrics and Killing tensors of the second order.

We first re-formulate the condition ``$g$ is geodesically equivalent to $\bar g$'' as a system of PDE. Consider  the following equation on  a symmetric $(0,2)$ tensor $L_{ij}$ (on  a manifold $(M,g)$)  
\begin{equation}
\label{eq:sinjukov} 
\nabla_k L_{ij}= \ell_{i} g_{jk}+ \ell_{j} g_{ik},
\end{equation} 
where $\ell_i$ is the differential of the function $\tfrac{1}{2} \sum_{j,k}g^{jk}L_{jk}$, $\ell_i = \tfrac{1}{2}
\tfrac{\partial }{\partial x^i} \sum_{j,k}g^{jk}L_{jk}$.

The equation  \eqref{eq:sinjukov} appeared independently and was used in many branches of mathematics, in particular in the theory of geodesically  equivalent metrics, see e.g. \cite{benenti, hyperbolic, lichnerowitz, iran, sinjukov}), and in the theory of integrable systems, 
see e.g. \cite{benenti0, benenti1, benenti2,  crampin, IIM, MR}. Later, e.g.  in
in \S  \ref{sec:killing}, we will use   the following statement:

\begin{Fact}[e.g., \cite{benenti}]  \label{fact:0}  Let $L=L^i_j$ be a $(1,1)$-tensor which is nondegenerate (i.e., $\det L\ne 0$),  self-adjoint with respect to $g$ and  such that $L_{ij} = g_{i\alpha}L^\alpha_j$ satisfies \eqref{eq:sinjukov}. 
Then the metric 
\begin{equation} \label{eq:bg} 
\bar g(\xi, \eta):= \tfrac{1}{\det L}\, g\left(L^{-1} (\xi), \eta\right)\end{equation}
is geodesically equivalent to $g$. Conversely,
 if two metrics $g$ and $\bar g$ are  geodesically equivalent, then the $(0,2)$-tensor  
\begin{equation} \label{eq:L}
L_{ij}:= \left(\frac{|\det(g)|}{ |\det(\bar g)| }\right)^{\tfrac{1}{n+1}} \sum_{s,r}  g_{is} \bar g^{sr} g_{rj}= \left(\frac{|\det(g)|}{ |\det(\bar g)| }\right)^{\tfrac{1}{n+1}}  g \bar g^{-1} g 
\end{equation}
satisfies \eqref{eq:sinjukov}. {\rm (}Note that  \eqref{eq:L}  is obtained from \eqref{eq:bg} by resolving it w.r.t. $L$.{\rm )}
\end{Fact} 

 Since  the addition of $\const\cdot \Id $ to $L=L^i_j$ does not affect  equation \eqref{eq:sinjukov}, we may assume that the  symmetric tensor $L_{ij}$  is nondegenerate.  This motivates the following definition, see e.g. \cite{splitting}: we say that a $(1,1)$-tensor $L^i_j$ is {\it geodesically compatible} with $g$, 
if it is $g$-self-adjoint and $L_{ij} = \sum_s g_{is}L^s_j$ satisfies \eqref{eq:sinjukov}.

Next, let us recall the relation of geodesically equivalent metrics  to separation of variables.

\begin{Fact} \label{fact:1} Suppose $L=L^i_j$ is geodesically compatible with $g$. Consider the  family of $(1,1)$-tensors,    polynomially  depending on  $t$,   given by    
\begin{equation}\label{eq:S1} 
S_t =  \left(t \cdot \Id_{n}-L \right)^{-1} \det\left( t\cdot \Id_{n}- L\right). 
\end{equation}
Then for any $t$, the symmetric $(0,2)$-tensor  $K_t( \xi, \eta) = g\bigl(S_t(\xi), \eta\bigr)$ is a Killing tensor for $g$. 
\end{Fact}

Note that by \cite{benenti}, geodesic compatibility with $g$ implies that the 
Nijenhuis  torsion of  $L=L^i_j$ vanishes.    In particular, if $L$ has $n$ different eigenvalues, a generic Killing tensor from the family $K_t$ satisfies the assumption in our definition of separation of variables. 

In the framework of geodesic equivalence, Fact \ref{fact:1} is due to \cite{MT,TM}; special cases of Fact \ref{fact:1}  were known already to P. Painlev\'e  \cite{Painleve1897} and T. Levi-Civita \cite{LC1896}. In the theory of integrable systems, Fact \ref{fact:1} (under additional nondegeneracy conditions) was discussed in e.g.  \cite{crampin, MR}.

The relation   between geodesically equivalent metrics,  Killing tensors  and separation of variables allows the  methods and results of one topic to be used when studying another. Killing tensors and corresponding  integrals were used in particular in \cite{hyperbolic, MT, lichnerowitz, TM} as additional tools to handle global behaviour of geodesically equivalent metrics. In the present paper, we use this relationship in the other direction:  we will apply the results and methods of the theory of geodesically equivalent metrics in the theory of separation of variables.

Let us   show that  geodesic compatibility naturally come in our description of separable coordinates: we show the existence of an (essentially unique) tensor  $L$ which is compatible with the  metric  \eqref{eq:hatg} and  diagonal in the coordinates $x^1,...,x^n$. First, we consider the case when $\mathsf F$ contains one vertex only; in this case, the metric is given by \eqref{eq:tildeg}.

\begin{Fact} \label{fact:2} Consider the metric $g=g_\alpha$ and the tensor $L= L_\alpha$ given by  \eqref{eq:tildeg}. Then, the tensor $L$ is geodesically compatible   with   $g$. 
\end{Fact}

This fact follows  directly   from \cite{LC1896}, where  Levi-Civita has proved that (for any functions $f_i(x^i)$  such that the formula below defines nondegenerate metrics) the metrics $g$ and $\bar g$ given by  
\begin{eqnarray*}
g &=& \sum_{i} \left(\prod_{s \ne i} \left( f_s(x^s)  -f_i(x^i)\right)\right) (dx^i)^2 \\
\bar g &=& \sum_{i} \left( \prod_{s \ne i}\left( \frac{f_s(x^s)  -f_i(x^i)}{f_i(x^i) \prod_{r}f_r(x^r)}  \right)\right) (dx^i)^2
\end{eqnarray*}
are geodesically equivalent.  We see that, for these metrics, 
 $L^i_j$ given by \eqref{eq:L} is $\textrm{diag}(f_1(x^1),...,f_n(x^n))$. 
The formulas  for $g$ and $L$  obtained  by  the 
 coordinate change  $x^i_{\rm{new}}= f_i(x^i)$  include our formulas  \eqref{eq:tildeg} for  $g_\alpha$, $L_\alpha$ as a special case. 

Now, let us describe  the tensor $L$ satisfying \eqref{eq:sinjukov} for the metric  $g$   given by \eqref{eq:hatg} with an arbitrary in-directed forest 
$\mathsf F$. We start with the case when  $\mathsf F$ is  connected. 

Let $\alpha$ be the root of $\mathsf F$ and assume that $\alpha$ has precisely $m$ incoming edges $\vec{\gamma_1\alpha}, ... ,  \vec{\gamma_m\alpha}.$  Next, we consider the following $(1,1)$-tensor $L$. It is blockdiagonal, the 
 $\alpha$-block is $L_\alpha$. For every $i$, the block corresponding to $\gamma_i$ is $\lambda_{\gamma_i} \Id_{n_{\gamma_i}}$. Next, for every $\beta$ with $\gamma_i\prec \beta$,  the block corresponding to $\beta$ is $\lambda_{\gamma_i} \Id_{n_{\beta}}$.

\begin{Fact} \label{fact:3} The $(1,1)$-tensor $L$ constructed above is geodesically compatible with  $g$.  Moreover, for any constants $A$ and $C$, the  $(1,1)$-tensor $\tilde L= AL + C\cdot \Id$ is  geodesically compatible with  $g$.
\end{Fact} 
This  fact follows  directly   from \cite{LC1896}: certain examples of geodesically equivalent metrics in this paper will give such $L$.

Let us now comment on the case when  $\mathsf F$   contains $m\ge 2$ connected components $\mathsf F_1,..., \mathsf F_m$. The formula \eqref{eq:hatg} immediately implies that  the  
 metric \eqref{eq:hatg}  is the direct product of metrics corresponding to its connected components. Indeed, if two vertices $\alpha\ne \beta$ satisfy neither 
$\alpha\prec \beta$ nor $\beta \prec \alpha$, then the components of $ g_\alpha$ do not depend on the coordinates  $X_\beta$ and the components of $ g_\beta$ do not depend on the coordinates  $X_\alpha$.  
Let us consider the $(1,1)$-tensor  $L$    corresponding to this direct product. That is, the tensor $L$  is block-diagonal with 
 $m$ blocks corresponding to the connected components of $\mathsf F$. Moreover,   
 the block  corresponding to the $i$th connected component of $\mathsf F$ is $C_i \Id$ with $C_i\in \mathbb{R}$ (we see  that the space of such operators 
$L$ is $m$-dimensional, in contrast to the  2-dimensional space  from Fact \ref{fact:3}).

\begin{Fact} \label{fact:4} The $(1,1)$-tensor $L$ constructed above is geodesically  compatible with  $g$. 
\end{Fact} 
Indeed, by construction it is self-adjoint and parallel with respect to $\nabla^g$ and \eqref{eq:sinjukov} is fulfilled.

\begin{Theorem} \label{thm:L}     
Assume a  self-adjoint $(1,1)$-tensor $\tilde L$ is geodesically compatible with
 $g$ given by \eqref{eq:hatg} and is diagonal in the coordinates  $x^1,...,x^n$.

Then, if the graph $\mathsf F$ is connected,  	
	then $\tilde L$   is given by $\tilde L = A L + C \cdot \Id$ for some $A, C\in \mathbb{R}$ and $L$ as in Fact \ref{fact:3}.

If 		$\mathsf F$   contains $m\ge 2$ connected components, then  $\tilde L = L$, where $L$   is as in Fact \ref{fact:4}.
			\end{Theorem}
			
Note  that in dimension  $n=1$ Theorem is wrong since  any tensor $L$ is geodesically compatible to $g$. 

The importance of the tensor $L$ in the theory of orthogonal separation of variables  was known before. It is the main ingredient of the approach of 
\cite{RM, RM2, RM3} and   was significantly used in particular in \cite{schoebel1,schoebel2}. We expect that   the essential uniqueness of this tensor will  provide additional tools in the study of separation of variables by methods of algebraic geometry, and will also improve the description of separating coordinates suggested in \cite{RM}. We will also use it in the present paper.

\subsection{Multiple warped product decomposition}  \label{sec:1.3.2.}

Consider the metric $g$ given by \eqref{eq:hatg}. Assume that the in-directed forest $\mathsf F$ is connected  and contains more than one vertex. Then $g$  has the  structure of  a warped product. 

 Indeed, for any 
vertex $\alpha$   of $\mathsf F$   which is not  the root,  
 the metric  $g$ naturally decomposes into   
the warped product such that  the  fibre  metric is the metric   
\eqref{eq:hatg} constructed by     the subgraph of  $\mathsf F$  
 spanned over  the vertex $\alpha$ and all vertices $\gamma$ with $\alpha\prec \gamma$.  We denote this metric by $\tilde g_\alpha$. 
 The labelled  in-directed tree corresponding to the base 
 metric is the part of  the in-directed tree 
 spanned over  all other  vertices. The warping function is  $f_\alpha$  defined by \eqref{eq:falpha} (or $1/f_\alpha$ if we consider covariant metrics).

\begin{Ex} \label{ex:beta2}{\rm 
For the graph $\mathsf F$ on Figure \ref{Fig:1} and  $\alpha=2$, 
the base metric is 
$
P_1(L_1) \cdot  g^{\LC}_1, 
$
the fibre metric   is     
$$
P_2(L_2) \cdot g^{\LC}_2  +  \tfrac{1}{
\det(\lambda_3 \Id_{n_2}-L_2)} \cdot P_3(L_3) \cdot g^{\LC}_3 +   
\tfrac{1}{\det(\lambda_4 \Id_{n_2}-L_2)}   P_4(L_4) \cdot g^{\LC}_4, 
$$
and the warping function is 
$
\tfrac{1}{\det(\lambda_2 \Id_{n_1}-L_1)}. 
$ 
}\end{Ex}

\begin{Ex} \label{ex:beta4} {\rm
  For the  same graph $\mathsf F$ on Figure \ref{Fig:1} and $\alpha=4$, 
the base metric is 
$  
 P_1(L_1) \cdot  g^{\LC}_1  + \tfrac{1}{
 \det(\lambda_2 \Id_{n_1}-L_1)}\cdot  P_2(L_2) \cdot g^{\LC}_2  + 
\tfrac{1}{\det(\lambda_2 \Id_{n_1}-L_1) \det(\lambda_3 \Id_{n_2}-L_2)} \cdot P_3(L_3) \cdot g^{\LC}_3,  $ 
 the fibre metric is 
 $   P_4(L_4) \cdot g^{\LC}_4, 
$
and the warping function is
$\tfrac{1}{\det(\lambda_2 \Id_{n_1}-L_1) \det(\lambda_4 \Id_{n_2}-L_2)}$. 
}\end{Ex}

Note that the base and  fibre metrics have the form \eqref{eq:hatg} and the  corresponding labelled in-directed trees are obtained from the initial in-directed tree  by deleting  the edge $\overrightarrow{\alpha\operatorname{next}(\alpha)}$.

\subsection{Proof of Theorem \ref{thm:L}} 

We take the metric $g$ given by \eqref{eq:hatg}  and 
 work in the separating coordinates $x^1,...,x^n$. We will first consider the case when $\mathsf F$ is connected. Assume   a diagonal tensor $\tilde L=\textrm{diag}(\tilde L^1_1,\tilde L^2_2,\dots ,\tilde L^n_n)$
 is geodesically compatible with $g=\textrm{diag}(g^{11},g^{22},\dots, g^{nn})$.  Our goal it to show that $\tilde L$ is as in the first  claim of  Theorem  \ref{thm:L}. 

Consider the Hamiltonian  $H= \tfrac{1}{2}\sum_{s} {p_s}^2 g^{ss}$ of the geodesic flow of  $g$ and  
the function $D:T^*M \to \mathbb{R}$ given by $D= \tfrac{1}{2} \sum_s  {p_s}^2 g^{ss} \tilde L^s_s$.   From the theory of geodesically equivalent metrics (e.g., from the Levi-Civita Theorem) it follows that the component $\tilde L^i_i$ may depend on the variable $x^i$ only and if the components  $\tilde L^i_i\equiv \tilde L^j_j$ (with  $i\ne j$) then both of them are constant. Since the equation  \eqref{eq:sinjukov} is linear in $\tilde L$, we may  assume without loss of generality 
that  the restriction of $\tilde L$ to the block corresponding to the root has simple eigenvalues which  are not constant  and 
 are different from any  eigenvalue of the restriction of $\tilde L$ to any other block. 

By \cite{benenti}, geodesic compatibility of such $\tilde L$ with $g$ is  equivalent to the  following (Ibort-Marmo-Magri) condition  
$$
\{ H, D\} = H \ \sum_{s,r} g^{sr} p_r\frac{\partial }{\partial x^s} \operatorname{trace}(\tilde L).  
$$
This condition is a homogeneous polynomial  of degree $3$ in the momenta $p_i$ whose coefficients depend  on  the  position. Equating the  $p_i^2p_j$-coefficients   and using  $\tfrac{\partial  \tilde L^i_i} {\partial x^j}=0$ for $i\ne j$ gives us  
\begin{equation} \label{eq:rastelli}  
   \frac{\partial g^{jj}}{\partial x^i} (\tilde L^i_i -\tilde L^j_j)= \frac{\partial \tilde L^i_i}{\partial x^i} g^{jj}.
\end{equation}
Recall that in this formula the components $g^{ii}$ are known and given by \eqref{eq:hatg} and  our goal is to show that $\tilde L$ is as we claim in Theorem     \ref{thm:L}. 

To this end, 
we  take the coordinate  $x^j$ from the block corresponding to the root  and $x^i$ from any other block. Then,  $\frac{\partial g^{jj}}{\partial x^i}=0$ and \eqref{eq:rastelli} implies that every  component
  $\tilde L^i_i$ such that $x^i$  is not from the block corresponding to the root is  constant.

Now, take  a vertex $\gamma$ such that $\textrm{next}(\gamma)$ is the  root 
 and  coordinates $x^i$, $x^j$   such that $x^j$ corresponds to the root block  and and $x^i$  corresponds 
 to  the  $\gamma$-block or to any $\beta$-block with $\gamma\prec \beta$.  
In this setup,   the equation \eqref{eq:rastelli} implies 
\begin{equation}\label{eq:rastelli1}
\tilde L^j_j- \tilde L^i_i= (x^j- \lambda_{\gamma})\frac{\partial \tilde L^j_j}{\partial x^j}. 
\end{equation} 
We see that     \eqref{eq:rastelli1}  determines uniquely $\tilde L^i_i$; in particular all $\tilde L^i_i$ such that   $x^i$  corresponds 
 to  the  $\gamma$-block or to any $\beta$-block with $\gamma\prec \beta$  are equal to each other (recall that  we already know that they are constants).   Moreover, \eqref{eq:rastelli1}  implies  that   $\frac{\partial \tilde L^j_j}{\partial x^j}=\const$ so    \eqref{eq:rastelli1} 
reads 
\begin{equation} \label{eq:rastelli1.5} \tilde L^j_j- \tilde L^i_i=(x^j- \lambda_{\gamma})\const \end{equation} 

Note that  the constant  in \eqref{eq:rastelli1.5}   may a  priori  be different for different
  $j$; let us show that it is not the case. 
Take $i\ne j$ such that $x^i$ and $x^j$ both belong to the root block. Equation \eqref{eq:rastelli} implies 
\begin{equation} 
\label{eq:rastelli2}  
  \tilde L^j_j- \tilde L^i_i= (x^j- x^i)\frac{\partial \tilde L^j_j}{\partial x^j} \ \ \textrm {  and  }  \ \  \tilde L^j_j- \tilde L^i_i= (x^j- x^i)\frac{\partial \tilde L^i_i}{\partial x^i}.
\end{equation}
Since $\frac{\partial \tilde L^j_j}{\partial x^j}$ depends only on $x^j$, and  $\frac{\partial \tilde L^i_i}{\partial x^i}$ depends only on $x^i$,  we have   $\frac{\partial \tilde L^j_j}{\partial x^j}=\frac{\partial \tilde L^i_i}{\partial x^i}$. Thus, 
    for every coordinate $x^i$ from the root block, we get     $\frac{\partial \tilde L^i_i}{\partial x^i}=\const$ and this constant is the same   for all coordinates 
		from the root block. By assumptions, it is  different from zero. 
Without loss of generality we can assume that  this constant is one,  in this case  $\tilde L^j_j = x^j+ \const_j$ and  
the equation \eqref{eq:rastelli1}  reads $(x^j + \const_j- \tilde L^i_i) = (x^j- \lambda_\gamma)$. Then,  the constants  $\const_j$ are the same for all 
$x^j$ from the root block, under the additional  assumption that  the number of blocks is $\ge 2$. 
We may assume then $\const_j=0$ (otherwise we replace $L$ by $\tilde L- \const_j\cdot \Id$).  Then,  in the notation of \eqref{eq:rastelli1}, $\tilde L^i_i= \lambda_\gamma$ and $\tilde L$ is precisely as we claimed in Theorem \ref{thm:L}.

It remains to consider the case  when we have the root block only; it has  dimension $\ge 2$.   We already know that $\tilde L^j_j = x^j+ \const_j$ and we need to show that all $\const_j$ are the same. The condition  \eqref{eq:rastelli2}   reads
$ (x^j- x^i + \const_j-\const_i)= (x^j-x^i)$ implying the claim. Theorem \ref{thm:L} is proved for connected $\mathsf F$. 

Now let us assume that $\mathsf F$ has $m\ge 2$ connected components. Then, $g$  is (locally isometric to) the direct product of the metrics $\tilde g_1,...,\tilde g_m$, where  $\tilde g_i$  is the metric  \eqref{eq:hatg} constructed by the $i$-th connected component. As we explain above, the components of $\tilde L^i_j$ may depend on its own coordinate only. Taking in account that the equation \eqref{eq:sinjukov} can be equivalently rewritten as 
$$
\operatorname{trace}_g\operatorname{-free}\ \operatorname{part\ of \ }\tilde L_{ij ,k} := \tilde L_{ij,k}- \tfrac{1}{2}\sum_s \left(\tilde L^{s}_{s,i} g_{jk}  +  \tilde L^{s}_{s,j} g_{ik}\right)=0, 
$$
    we obtain that the restriction of $\tilde L$  to any components of the direct product is well-defined and satisfies \eqref{eq:sinjukov} with respect to the corresponding metric. Then, it is as in the first (already proved) claim of Theorem  \ref{thm:L}, that is, the restriction of $\tilde L$ to the $i$-th component of the direct product  is given by   		$A_i \tilde L_{i} + C_i\cdot \Id,$  where $\tilde L_i$ is as in Fact \ref{fact:3}. Next,  all $A_i= 0$. Indeed, otherwise the right hand side of \eqref{eq:sinjukov} has a nonzero components of the form $\tilde \ell_i g_{jj}$ with the coordinate $x^j$ belonging to another connected component of $\mathsf  F$. The 
corresponding component of the right hand side of \eqref{eq:sinjukov} is  clearly zero, which gives us a contradiction and proves the second claim of Theorem \ref{thm:L}.

\subsection{Proof of Theorem \ref{thm:2}} \label{sec:proof2}

The proof is based on Theorem \ref{thm:L}. The idea is as follows.  The tensor $L$ satisfying \eqref{eq:sinjukov} is a geometric object. By Theorem \ref{thm:L}, it  is essentially unique. If
$\mathsf F$ is not connected, one can  reconstruct the  metrics corresponding to different connected components of $\mathsf F$ by $L$ since they correspond to different constant eigenvalues of $L$. If $\mathsf F$ is connected,     $L$  gives us 
 coordinates corresponding 
to the root of $\mathsf F$, since they are nonconstant  eigenvalues of $L$.  

The formal proof goes by induction in the dimension. The base of induction is $n=1$ and is clear. Let us assume that Theorem \ref{thm:2} is correct for  dimensions $\le n-1$ and prove it for dimension $n$. 
Suppose two sets of admissible data describe the same separating coordinates on $M^n$,
 we will indicate  objects corresponding to the second set by ``tilde'', e.g., the metric \eqref{eq:hatg}
constructed by the second set is $\tilde g$ and the coordinates are $\tilde x^i$. 
That is, there exists an     isometry  between $g$ and $\tilde g$
 of  the ``diagonal'' form $x^i = x^i\left(\tilde x^{b(i)}\right)$ where   $b:\{1,...,n\} \to \{1,...,n\}$ is a bijection. Let us consider the pullback of $\tilde L$ with respect to this  isometry. It is diagonal in coordinates $x^i$ and satisfies \eqref{eq:sinjukov}.

 If $\tilde{\mathsf F}$ is not connected, by Theorem \ref{thm:L} 
 the pullback of $\tilde L$ 
 is parallel  and has $m\ge 2$ eigenvalues $\tilde C_i$  which implies that the number of connected  components of $\mathsf F$ and 
$\tilde{\mathsf F}$ coincide and that the isometry preserves  the decomposition in the direct product.
 Clearly, the metrics of the components of the direct product 
are the metrics of the form \eqref{eq:hatg}  corresponding to the connected components of $\mathsf F$.   Thus, we made  the induction step  by  reducing  
the case to low-dimensional cases  under the assumption that  $\tilde{\mathsf F}$ is not connected. 

Now, consider the  case when $ \mathsf F$ and $\tilde{\mathsf F}$ are    connected. Then, the coordinates corresponding to the root are   the eigenvalues  of $L$ from Fact \ref{fact:3}. Since  the pullback of 
$\tilde L$ is $AL +C \cdot \Id$ by Theorem \ref{thm:L}, for 
the coordinates $ x^i$ corresponding to the root,  the coordinates $\tilde x^{b(i)}$ also correspond to the root and the coordinate transformation for such coordinates is given by $A \tilde x^{b(i)} + C = x^i$. Moreover,   
 by Fact \ref{fact:3}  the labels $\lambda_\gamma$ staying on all incoming edges to the root are  constant eigenvalues of $L$. These implies that 
the root coordinates and the labels are changed as in (B) of \S \ref{sec:1.2}.  Note also that the isometry sends  the eigenspaces of $  L$  corresponding to constant eigenvalues of $L$ to that of $\tilde L$.
 
 Next, delete the root and   all incoming edges of $\mathsf F$ and of $\tilde {\mathsf F}$.
 As explained in \S \ref{sec:1.3.2.}, each of the  connected  components 
 of  the obtained labelled in-directed forest  gives a fibre metric of the form \eqref{eq:hatg} of the corresponding warped product decomposition. This fibre metric ``lives'' on the integral submanifold of the corresponding constant eigenvalue of $L$, so diagonal isometry of the metric induces the diagonal isometry of the  fibre metrics. This reduces the situation to that already discussed,  i.e., to  diagonal isometries of metrics of the form \eqref{eq:hatg} living on a 
  manifold of smaller dimension (and such that the corresponding in-directed forest is connected), which performs the induction step  in this case and therby    proves Theorem \ref{thm:2}.

\subsection{Proof of Theorem \ref{thm:3}} \label{sec:thm3}

We will use the relation to geodesic equivalence discussed  in \S \ref{sec:1.3.1}.  Take a vertex $\alpha$ and assume it has $m$ incoming edges.  Following  \S \ref{sec:1.3.2.}, consider the corresponding warping decomposition of the metric $g$ given by \eqref{eq:hatg}.  
Combining  Facts    \ref{fact:3} and  \ref{fact:1} we obtain   a $n_\alpha+m$-dimensional space of Killing tensors for the fibre metric.  Next,
 observe that  any Killing tensor for fibre metric, viewed as (2,0)-tensor on $M$, 
 is also a Killing tensor of the whole warped product metric\weg{(cf \cite[Proposition 4.3]{MR})}.  The fact is general and is true for any warped product and for Killing tensors of any degree; a ``brute force'' proof goes through calculation of the Christoffel symbols (as functions of the Christoffel symbols 
of the base and the fibre metrics and the derivatives of the warping function, see e.g. \cite{Prvanovic}) and inserting them in the Killing equation.  In our  case  of constant curvature, a shorter proof is available: it is   based  on 
 \cite{milson} where it was shown that 
  in  a constant curvature space, any Killing tensor viewed as 
tensor with upper indices is a polynomial in Killing vectors. Clearly,  
 Killing vectors for the fibre metric    can be lifted to    Killing vectors of the whole space which also implies that Killing tensors can be lifted to the whole space.

   Thus, any vertex $\alpha$ gives us an explicit $n_\alpha+m$-dimensional space of Killing tensors of $g$.  

Finally, we observe that this   family of  Killing   tensors contains an
 $n_\alpha$-dimensional subfamily   generated by \eqref{eq:K1}. This  proves Theorem \ref{thm:3}.

\section{Proof of Theorems \ref{thm:casimirs},  \ref{thm:generalised_casimirs} and  \ref{thm:F_casimirs}}

\subsection{Proof of Theorems  \ref{thm:casimirs}  and \ref{thm:generalised_casimirs}}\label{sect:proof1.5-1.6}

To prove \ref{thm:casimirs}  and \ref{thm:generalised_casimirs}  we simply compute the components $g^{ij}$ of the metric $g=P(L)g_{\mathsf{LC}}$ in coordinates $Y^j$'s  defined by \eqref{eq:10}  and  \eqref{eq:11} respectively.  Namely,  in the case of Theorem  \ref{thm:casimirs} we will show that the $n\times n$ matrix whose $(i,j)$-element is given by $g^*(\dd Y^i, \dd Y^j)$ has constant entries and is nondegenerate. Similarly, for Theorem \ref{thm:generalised_casimirs}  we will show that $g^*(\dd Y^i, \dd Y^j) = -K Y^i Y^j + G^{ij}$ where $G^{ij}$ is constant and non-degenerate  (see Remark \ref{rem:1.2}).

Let us compute $g^* ( \dd Y^\alpha_i, \dd Y^\beta_j)$   (including the case $i=\infty$  and/or $j=\infty$).   We start with the case when $i$ and $j$ are finite. Below, we assume that $P(t) = a_{n+1} t^{n+1} + \dots$, where $a_{n+1}$ might be zero or not, which allows us to prove the both theorems simultaneously.

Since $\tfrac{\dd}{\dd t}$ and $\dd=\dd_x$ (differential in the variables $x$) commute, we have 
\begin{equation}
\label{eq:bols_1}
g^* ( \dd Y^\alpha_i, \dd Y^\beta_j) = \tfrac{1}{\alpha!} \tfrac{1}{\beta!} \tfrac{\dd^\alpha}{\dd s^\alpha}|_{s=\mu_i} \tfrac{\dd^\beta}{\dd t^\beta}|_{t=\mu_j}  g^* \bigl(\dd Y(s),  \dd Y(t)\bigr),
\end{equation}
where $Y(t) = \sqrt{\det (t\cdot\Id - L)}$.
It is easily seen that:
$$
g^* \bigl(\dd Y(s),  \dd Y(t)\bigr) = g^*\left(\dd \sqrt{\det ( s\cdot \Id - L)},  \dd\sqrt{\det ( t\cdot \Id -L)}\right)  =   
$$
$$
\frac{1}{4} Y(s)Y(t) \sum_{i=1}^n \frac{P(x_i)}{(x_i-s)(x_i-t)\prod_{j\ne i}(x_i - x_j)}
$$

If we think of $s$ and $t$ as two additional variables like $x_{n+1}, x_{n+2}$, then adding two {\it auxiliary} terms  $ \frac{P(s)}{(s-t)\prod_{j}(s - x_j)}$  and  
$\frac{P(t)}{(t-s)\prod_{j}(t - x_j)}$ allows us to use the following famous algebraic identity: 
$$
\sum_{i=1}^{n+2} \frac{P(x_i)}{\prod_{s\ne i} (x_i - x_j)} = a_{n+1}
$$
for any polynomial $P(t)= a_{n+1} t^{n+1} + \dots$ of degree $n+1$. Hence,  the above expression becomes
$$
\frac{1}{4} Y(s)Y(t)  \left(  a_{n+1} -\frac{P(t)}{(t-s)\prod_{j}(t - x_j)}  - \frac{P(s)}{(s-t)\prod_{j}(s - x_j)}\right) =
$$
$$
\frac{1}{4} a_{n+1} Y(s)Y(t) - \frac{1}{4} \left(\frac{P(t)}{t - s} \cdot \frac{Y(s)  }{Y(t)} + \frac{P(s)}{s-t} \cdot \frac{Y(t)}{Y(s)} \right)=
$$
$$
\frac{1}{4} a_{n+1} Y(s)Y(t) - \frac{1}{4} \left( \frac{P(t)-P(s)}{t - s} - \frac{Y(t)-Y(s)}{t-s} \left(
\frac{P(t)}{Y(t)} + \frac{P(s)}{Y(s)}\right)\right)
$$

\weg{
$$
\frac{1}{4} a_{n+1} Y(s)Y(t)  - \frac{1}{4}  \left( \frac{P(t)}{t - s} \cdot h(t,s) + \frac{P(s)}{s-t} \cdot h(s,t) \right),
$$
where $h(t,s) = \frac{Y(s)}{Y(t)}$. The only important property of $h(t,s)$ is that $h(t,t) = 1$.  
}

Hence, in view of \eqref{eq:bols_1}, we have 
$$
g^* ( \dd Y^\alpha_i, \dd Y^\beta_j) = 
\tfrac{1}{\alpha!} \tfrac{1}{\beta!}\tfrac{\dd^\alpha}{\dd s^\alpha}|_{s=\mu_i} \tfrac{\dd^\beta}{\dd t^\beta}|_{t=\mu_j} \ 
\left(\tfrac{1}{4} a_{n+1}  Y(s)Y(t)- \tfrac{1}{4}  \left( \tfrac{P(t)-P(s)}{t - s} - \tfrac{Y(t)-Y(s)}{t-s} \left(
\tfrac{P(t)}{Y(t)} + \tfrac{P(s)}{Y(s)}\right)\right)\right)
 =
$$
$$
\tfrac{1}{4} a_{n+1}  Y_i^\alpha Y_j^\beta    
 - \tfrac{1}{4} \tfrac{1}{\alpha!} \tfrac{1}{\beta!}\tfrac{\dd^\alpha}{\dd s^\alpha}|_{s=\mu_i} \tfrac{\dd^\beta}{\dd t^\beta}|_{t=\mu_j} \ 
\left( \tfrac{P(t)-P(s)}{t - s} - \tfrac{Y(t)-Y(s)}{t-s} \left(
\tfrac{P(t)}{Y(t)} + \tfrac{P(s)}{Y(s)}\right)\right) 
$$
under the condition that $\mu_i$ and $\mu_j$ are zeros of $P$ of order at least $\alpha+1$ and $\beta+1$ respectively. Hence, the derivative of the last term is zero.  Moreover,   the derivative of  $\tfrac{P(t)-P(s)}{t - s}$ equals zero also if $i \ne j$.   Taking into account that the curvature of $g$ equals $K = -\frac{1}{4} a_{n+1}$ we get  $g^* ( \dd Y^\alpha_i, \dd Y^\beta_j) = - K \, Y_i^\alpha Y_j^\beta$ for $i\ne j$.

\weg{
For instance for $\alpha=0$ and $\beta=0$ we get for $i \ne j$:
$$
\left(\tfrac{P(t)}{t - s} \cdot h(t,s) + \tfrac{P(s)}{s-t} \cdot h(s,t)\right)|_{s=\mu_i, t=\mu_j} = 0 
$$
since $P(\mu_i)=0$ and $P(\mu_j)=0$. 
Using similar arguments it is easy to show that   for $i \ne j$ we will always get zero.  
}

\weg{
However, when $i=j$ we get:
$$
\left(\tfrac{P(t)}{t - s} \cdot h(t,s) + \tfrac{P(s)}{s-t} \cdot h(s,t)\right)|_{t=\mu_i, s=\mu_i}  = \tfrac{P(t)-P(s)}{t-s}|_{t=\mu_i, s=\mu_i} = P'(\mu_i).  
$$

So that
$$
g^*(Y_i^0, Y_i^0) = \tfrac{1}{4}Y_i^0 Y_i^0 - \tfrac{1}{4} P'(\mu_i)\quad\mbox{and}\quad
g^*(Y_i^0, Y_j^0) = \tfrac{1}{4}Y_i^0 Y_j^0.
$$

In particular, in the simplest case when $\deg P = n$ and $\mu_1,\dots,\mu_n$ are simple roots of $P$, we see that the functions  $Y^0_1, \dots, Y^0_n$ are flat coordinates for $g$ and the matrix $(g^{ij})$ in these coordinates has the form
$$
 \operatorname{diag}\left(  -\tfrac{1}{4} P'(\mu_1), \dots, -\tfrac{1}{4} P'(\mu_n)\right),
$$
which coincide with the classical for formulas for ellipsoidal  coordinates in $\R^n$.

Similarly,  if $\deg P = n+1$ and $\mu_1,\dots,\mu_{n+1}$ are simple roots of $P$.  Then $Y_i^0=Y(\mu_i)$ are flat coordinates for $g$ and the matrix $g^{ij}= g^*(dY_i^0,dY_j^0)$ has the form     
$$
\operatorname{diag}\left(  -\tfrac{1}{4} P'(\mu_1), \dots,  -\tfrac{1}{4}  P'(\mu_{n+1})\right)  + \tfrac14 a_{n+1} Y Y^\top,  \quad \mbox{where $Y = (Y_1^0 \, \dots \, Y_{n+1}^0)^\top$}.
$$
}

Now assume that $i = j$ and $\mu_i$ is a root of multiplicity $k_i$.    W. l. o. g.  we may assume that $\mu_i = 0$. Then $0$ is a root of $P$ of some order $k_i$, $\alpha,\beta < k_i$, and we can write $P(t) = -4 \,c \, t^{k_i} \wt P_i(t) $ with $\wt P_i=\prod_{s\ne i} (t-\mu_s)^{k_s}=a_0 + a_1 t + a_2 t^2 + \dots$.  Then we have:
$$
- \tfrac{1}{4} \tfrac{1}{\alpha!} \tfrac{1}{\beta!}\tfrac{\dd^\alpha}{\dd s^\alpha}|_{s=\mu_i} \tfrac{\dd^\beta}{\dd t^\beta}|_{t=\mu_j} \ 
\left( \tfrac{P(t)-P(s)}{t - s}\right)=
c\, \tfrac{1}{\alpha!} \tfrac{1}{\beta!}\  \tfrac{\dd^\alpha}{\dd s^\alpha} \tfrac{\dd^\beta}{\dd t^\beta}|_{t=0, s=0} \tfrac{t^{k_i} \wt P_i(t)-s^{k_i} \wt P_i(s)}{t - s} = 
$$
$$  
\tfrac{1}{\alpha!}  \tfrac{1}{\beta!}\ \tfrac{\dd^\alpha}{\dd s^\alpha} \tfrac{\dd^\beta}{\dd t^\beta}|_{t=0, s=0} \left(  a_0 (t^{k_i-1} + t^{k_i-2} s +\dots + s^{k_i-1}) + a_1 (t^{k_i} + t^{k_i-1} s +\dots + s^{k_i})     \right) = 
$$
$$
= c \cdot a_k = c \cdot \tfrac{1}{k!}\tfrac{\dd^k}{\dd t^k}{}_{|\,t=\mu_i} \wt P_i (t), \quad \mbox{where   $k=\alpha+\beta - k_i + 1$.}
$$
Summarising, we obtain (recall that $\alpha < k_i$, $\beta < k_j$):
\begin{itemize}
\item  $\deg P \le n$:
\begin{equation}
\label{eq:bols_9}
g^*(\dd Y^\alpha_i, \dd Y_j^\beta) = \begin{cases}
0, & \mbox{\small if $i\ne j$}\\
c\cdot \wt P_i^{(k)}, & \mbox{\small if $i= j$ and $k = \alpha+\beta - k_i +1\ge 0$.}\\
\end{cases}
\end{equation}
\item $\deg P = n+1$:
\begin{equation}
\label{eq:bols_10}
g^*(\dd Y^\alpha_i, \dd Y_j^\beta) = \begin{cases}
- K\ Y^\alpha_i Y_j^\beta , & \mbox{\small if $i\ne j$},\\
- K \left( Y^\alpha_i Y_j^\beta - \wt P_i^{(k)}\right), & \mbox{\small if $i= j$ and $k = \alpha+\beta - k_i +1\ge 0$},\\
\end{cases}
\end{equation}
\end{itemize}
where  $\wt P_i^{(k)} =  \tfrac{1}{k!} \tfrac{\dd ^{k}}{\dd t^k}{}_{|\, t=\mu_i} \wt P_i(t)$, 
$\wt P_i(t)=\prod_{s\ne i}{(t-\mu_s)^{k_s}}$ and $K=c$ is the curvature of $g$.

In matrix form, this means that the block of $G$ that corresponds to the functions 
$Y^0_i, \dots, Y^{k_i-1}_i$
takes the form 
$$
c\cdot \begin{pmatrix}
0& \dots &0 & 0&  \wt P_i^{(0)} \\
\vdots & \iddots & \iddots &  \wt P_i^{(0)} &  \wt P_i^{(1)}\\
0 & 0 &\iddots & \iddots& \wt P_i^{(2)}\\
0 &  \wt P_i^{(0)} &  \wt P_i^{(1)}&\iddots &\vdots \\
  \wt P_i^{(0)}&  \wt P_i^{(1)}&  \wt P_i^{(2)} &\dots  &  \wt P_i^{(k_i-1)}
\end{pmatrix},
$$
which coincides with the description given in Remarks \ref{rem:matrix} and \ref{rem:matrix1}.

For the ``infinite'' root $\mu=\infty$, the proof is similar.  Recall that in this case $\deg P \le n$ and $Y^\alpha_\infty$ is defined by:
$$
Y_{\infty}^{\alpha}:= \tfrac{1}{\alpha!}\tfrac{\dd^\alpha}{\dd t^\alpha} \sqrt{\det(\Id - t\cdot {L})}_{|t=0}    - 
R^{(\alpha)},   \quad   \alpha=1,...,d=n-\deg P.
$$
For our further purposes,  we will also treat the case $\alpha = d+1$ with a slightly modified formula:
$$
Y_{\infty}^{d+1}:= \tfrac{1}{\alpha!}\tfrac{\dd^{d+1}}{\dd t^{d+1}} \sqrt{\det(\Id - t\cdot {L})}_{|t=0}    - 
\tfrac{1}{2}R^{(d+1)}.
$$
The constants $R^{(\alpha)}$ are defined as follows
$R^{(\alpha)} = \tfrac{1}{\alpha!} \tfrac{\dd^\alpha}{\dd t^\alpha}  \prod_{i=1}^m (1 - t\mu_i)^{k_i}{}_{|t=0}$.

As above, we will compute $g^* (\dd Y_i^\alpha, \dd Y^\beta_\infty)$ and $g^* (\dd Y_\infty^\alpha, \dd Y^\beta_\infty)$ as the derivatives:
\begin{equation}
\label{eq:bols_2}
g^* (\dd Y_i^\alpha, \dd Y^\beta_\infty) = \tfrac{1}{\alpha!} \tfrac{1}{\beta!} \tfrac{\dd^\alpha}{\dd s^\alpha}{}_{|\, s=\mu_i}\tfrac{\dd^\beta}{\dd t^\beta}{}_{|\,t=0} \ g^* (Y(s), \wt Y(t)), 
\end{equation}
and
\begin{equation}
\label{eq:bols_3}
g^* (\dd Y_\infty^\alpha, \dd Y^\beta_\infty) = \tfrac{1}{\alpha!} \tfrac{1}{\beta!} \tfrac{\dd^\alpha}{\dd s^\alpha}{}_{|\, s=0}\tfrac{\dd^\beta}{\dd t^\beta}{}_{|\,t=0} \ g^* (\wt Y(s), \wt Y(t)), 
\end{equation}
where $Y(s) = \sqrt{\det(s\cdot \Id - L)}$ and $\wt Y(t) = \sqrt{\det(\Id - t\cdot L)}$.
Since all the computations are quite similar to the case of finite roots $\mu_i$, we only indicate the most essential steps:
$$
g^*(\dd Y(s), \dd \wt Y(t)) = \frac{1}{4} Y(s) \wt Y(t) \sum_{i=1}^n \frac{P(x_i)}{(s-x_i)(t^{-1} - x_i)\prod_{i\ne j} (x_i - x_j)}=
$$ 
$$
-\frac{1}{4} Y(s) \wt Y(t) \left(  \frac{P(s)} {(s-t^{-1}) \prod_j (s - x_j)   } + \frac{P(t^{-1})}{(t^{-1}-s) \prod_j (t^{-1} - x_j)   }          \right)=
$$ 
$$
-\frac{1}{4} \left(  \frac{t\wt Y(t)} {(st-1)} \cdot \frac{P(s)}{Y(s)}+ \frac{t^{n+1} P(t^{-1})}{(1-st)}   
\frac{Y(s)}{\wt Y(t)}           \right).
$$

Computing the derivative \eqref{eq:bols_2} gives:
\begin{equation}
\label{eq:bols_6}
g^* (\dd Y_i^\alpha, \dd Y^\beta_\infty) = \begin{cases}
0,  & \mbox{if $\beta \le d$,}\\
c\cdot Y^\alpha_i, & \mbox{if $\beta = d+1$,}
\end{cases}
\end{equation}

Similarly, 
$$
g^*(\dd \wt Y(s), \dd \wt Y(t)) =   \frac{1}{4} \wt Y(t) \wt Y(s) \sum_{i=1}^n \frac{P(x_i)}{(t^{-1} - x_i)(s^{-1} - x_i)\prod_{i\ne j} (x_i - x_j)}=
 $$
  $$
-\frac{1}{4} \wt Y(t) \wt Y(s) \left(  \frac{P(t^{-1})} {(t^{-1}-s^{-1}) \prod_j (t^{-1} - x_j)   } + \frac{P(s^{-1})}{(s^{-1}-t^{-1}) \prod_j (s^{-1} - x_j)   }          \right)=
$$ 
$$
c \cdot st  \left(\frac{ t^{d}\, \wt P_\infty(t) - s^{d}\, \wt P_\infty(s) }{s-t} +
\left(\frac{t^{d}\, \wt P_\infty(t)}{\wt Y(t)} + \frac{s^{d}\, \wt P_\infty(s)}{\wt Y(s)}  \right) \cdot  \frac{\wt Y(s)-\wt Y(t)}{s-t}  \right), 
$$
where $\wt P_\infty (t) = \prod_{i=1}^m (1 - t\mu_i)^{m_i}$.

Computing the derivative \eqref{eq:bols_3} gives (notice that for the l.h.s of the below relation the constant term $R^{(\alpha)}$ in the definition of  $Y^\alpha_\infty$ does not play any role, but for the r.h.s. it does!):
\begin{equation}
\label{eq:bols_7}
g^*(\dd Y^\alpha_\infty, \dd Y_\infty^\beta) = \begin{cases} 
-c\cdot R^{(k)}, & \mbox{\small if $\alpha,\beta<d+1$ and $k = \alpha+\beta-(d+1)$}, \\
\phantom{-}c\cdot   Y_\infty^\beta\, , & \mbox{\small if $\alpha= d+1, \beta< d+1$},\\
\ 2 c\cdot Y^{d+1}_\infty, & \mbox{\small if $\alpha=  \beta= d+1$}.\\
\end{cases}
\end{equation}

Summarising the relations from \eqref{eq:bols_9}, \eqref{eq:bols_10}, \eqref{eq:bols_6}, \eqref{eq:bols_7} we obtain the statements of Theorems \ref{thm:casimirs} and  \ref{thm:generalised_casimirs}. More precisely,  
Theorem \ref{thm:generalised_casimirs} follows immediately from \eqref{eq:bols_10}, whereas Theorem \ref{thm:casimirs} follows from \eqref{eq:bols_9}, the first case of \eqref{eq:bols_6} and the first case of \eqref{eq:bols_7}.   These formulas also imply that the constant matrix $G$ related to 
our (generalised) flat coordinates (see Remark \ref{rem:1.2}) is block-diagonal.  The $i$-th block of $G$ corresponds to the $i$-th root $\mu_i$ of $P(t)$  (or, equivalently, to the functions $Y_i^0,\dots, Y_i^{k_i-1}$). If $d = n-\deg P>0$, then there is one more block related to the functions $Y_\infty^1,\dots, Y_\infty^{d}$ (and corresponding to the `infinite' root).   The structure of these blocks is as described in Remarks \ref{rem:matrix} and \ref{rem:matrix1}.

For our further purposes we will introduce the function $\phi = \frac{1}{c} Y^{(d+1)}_\infty$  (cf. formula \eqref{eq:bols_8}).  It follows from \eqref{eq:bols_6} and \eqref{eq:bols_7} that $\phi$ satisfies the following  
relations:
$$
g^*(\dd \phi, \dd Y_i^\alpha) = Y_i^\alpha   \quad\mbox{(including $i=\infty$)\quad and}\quad
g^*(\dd\phi, \dd\phi) = 2\phi,
$$
which will be used essentially in the next section.

\subsection{Proof of Theorem \ref{thm:F_casimirs}}

We start with the case of a warp product (contravariant) metric
$$
g(u,v)=g_1(u)  +  \tfrac{1}{f^2(u)} \, g_2(v),
$$
where $g_1$ and $g_2$ are constant curvature metrics (with curvatures $K_1$ and $K_2$ respectively).
Our first goal is to construct (generalised) flat coordinates for $g$ ``from''  (generalised) flat coordinates $u=(u^1, u^2, \dots )$ and $v=(v^1, v^2, \dots)$ on the $g_1$- and $g_2$-blocks.  Recall that (generalised) flat coordinates are characterised by the relation
$$
g_1 (\dd u^i, \dd u^j) = (G_1)^{ij}  - K_1 u^i u^j,  \quad\mbox{and similarly}\quad
g_2 (\dd v^k, \dd v^m) = (G_2)^{km}  - K_2 v^k v^m
$$ 
where $G_1$  and $G_2$ are constant matrices.  It is a well-known fact that in order for $g$ to have constant curvature, the function $f$ must satisfy the following two conditions: (1)  $f=f(u)$ is a (generalised) flat coordinate for $g_1$ and (2) $g_1^*(\dd f,\dd f) = K_2 - K_1 f^2$. We will use these conditions in our computations below.  In the context of Theorem \ref{thm:F_casimirs}, they are fulfilled by construction. 

First assume that $K_2 \ne 0$.   Let $u^1, \dots, u^p$ and $v^1, \dots, v^{n_2+1}$ be (generalised) flat coordinates of the first and second blocks respectively (here $p$ is either $n_1$ or $n_1+1$ depending on whether $g_1$ is flat or not).
W.l.o.g.  we assume that $f=u^{1}$  and  $u^2,\dots, u^p$   span the orthogonal complement of $f=u^1$. In other words, 
 $$
 G_1 = \begin{pmatrix}  c & 0 \\ 0 &\tilde G_1   \end{pmatrix}  \quad\mbox{with  $c\ne 0$}.
 $$
 \begin{Proposition}\label{prop:bols1} 
 The functions $u^2, \dots, u^{p},    fv^1, \dots fv^{n_2+1}$ are generalised flat coordinates for $g$. The corresponding matrix $G$ w.r.t. these coordinates takes the following form
 $$
 G =  \begin{pmatrix}   \tilde G_1 & 0 \\ 0 & G_2  \end{pmatrix}.  
 $$ 
 \end{Proposition}
 
 \begin{proof}
 We need to verify three relations:
 $$
 \begin{aligned}
 \mathrm{(i1)}: \ \ & g^* (\dd u^i, \dd u^j) =  (\tilde G_1)_{ij} - K_1 u^i u^j, \quad i,j \ge 2;\\
 \mathrm{(i2)}: \ \ &g^* (\dd u^i, \dd (fv^k)) =   - K_1 u^i (f v^k), \quad i \ge 2;\\
 \mathrm{(i3)}: \ \ &g^* (\dd (fv^k), \dd (fv^m)) = (G_2)_{km}  - K_2 (fv^k) (f v^m).
 \end{aligned}
$$
Relation (i1) is obvious as  $g^* (\dd u^i, \dd u^j) =  g_1^* (\dd u^i, \dd u^j)$. Next we have
$$
g^* (\dd u^i,  \dd(f v^k)) = g^*( \dd u^i,   v^k \dd f + f \dd v^k) = g_1^* (\dd u^i , \dd f) v^k = g_1^* (\dd u^i , \dd u^1) v^k =
(0 - K_1 u^i f) v^k,
$$
as needed.  Finally, we compute  (using the additional condition $g_1^* (\dd f , \dd f) = K_2 - K_1 f^2$):
 $$
 g^* ( \dd (f v^k), \dd (fv^m)) = g^* ( v^k \dd f {+} f \dd v^k,  v^m \dd f {+} f \dd v^m) = g_1^* (\dd f , \dd f) v^k v^m + \tfrac{1}{f^2} g^*_2 (f \dd v^k, f \dd v^m) =
 $$
 $$
 (K_2 - K_1 f^2) v^k v^m + \left( (G_2)_{ij} - K_2 v^k v^m\right) = (G_2)_{ij} - K_1 (fv^k) (fv^m),
 $$
 as required.
 \end{proof}
 
Next, we assume $K_2 =0$.   In this case,  $G_2$ is of size $n_2\times n_2$.  W.l.o.g. we assume that the `orthogonal complement' to $f = u^1$ (in the sense of $G_1$) is spanned by $u^1, \dots, u^{p-1}$  and  $g_1^* (\dd f, \dd u^p) = b\ne 0$ so that $(G_1)^{1j}=0$ for all $j\ne p$. 
 
In addition to the flat coordinates $v^1, \dots, v^{n_2}$ and (generalised) flat coordinates $u^1,\dots, u^p$, we consider a function $\phi$ on the second block such that
$g_2^* (\ddd\phi, \ddd v^i) = v^i$ and $g_2^* (\ddd\phi, \ddd \phi) = 2\phi$.

\begin{Proposition}\label{prop:bols2}
The functions 
\begin{equation}
\label{eq:bols_4}
u^1, \dots, u^{p-1}, \  \tilde u^{p} = u^{p} - b\cdot f\phi, \  fv^1,\dots, f v^{n_2}
\end{equation}
form a collection of independent (generalised) flat coordinates for $g$.

The corresponding matrix $G$  w.r.t. these coordinates takes the form 
$$
G = \begin{pmatrix}  G_1 & 0 \\ 0 & G_2 \end{pmatrix}.
$$
\end{Proposition}

\begin{proof} The components of $G$ related to the ``new'' flat coordinates are all as expected (the proof is literally the same as in previous proposition), except perhaps for those related to the new function   $\tilde u^p = u^p - f\phi$. Below we compute the inner product of $\dd\tilde u^p$ with the differentials of each of the functions \eqref{eq:bols_4}.

We first compute $g^*(\dd\tilde u^p, \dd u^i)$,  $i\ne p$: 
$$
g^*(\dd \tilde u^p, \dd u^i) = 
 g^*( \dd u^p -  b\phi \dd f - bf \dd \phi,  \dd u^i)=g_1^*( \dd u^p ,  \dd u^i) - b\phi \ g_1^*( \dd f ,  \dd u^i)  = 
$$
$$
(G_1)_{pi} - K_1 u^p u^i + b \phi \cdot K_1 f u^i   =
(G_1)_{pi} - K_1 \tilde u^p u^i,
$$
as needed. Next we have
$$
g^*\left(\dd\tilde u^p, \dd(fv^j)\right) = g^*\left(  \dd u^p   - b\phi \dd f  - bf \dd \phi)   , f \dd v^j + v^j \dd f\right) =
$$ 
$$
g_1^*(\dd u^p,  v^j \dd f) - g_1^* (b\phi \dd f,  v^j \dd f)  -  b g_2^* (\dd \phi, \dd v^j)= 
$$
$$
v^j \, g_1^*(\dd u^p,  \dd f)  + K_1 b \phi v^j f^2 -  b   g_2^* (\dd \phi, \dd v^j)= 
$$
$$
v^j ( b - K_1 u^p f)  + K_1 b\phi v^j f^2 - b v^j  =
- K_1 f v^j (u^p -b f\phi) = - K_1\tilde u^p  (f v^j),
 $$
as needed. Finally, 
 $$
 g^*\left(\dd\tilde u^p, \dd\tilde u^p\right) = g^*(\dd u^p   - b\phi \dd f  - bf \dd\phi , \dd u^p   - b\phi \dd f  - bf \dd\phi) =
 $$
 $$
 g_1^* (\dd u^p   - b\phi \dd f , \dd u^p   - b\phi \dd f ) + b^2 g_2^* (\dd\phi, \dd\phi) =
 $$
 $$
 (G_1)_{pp} - K_1 (u^p - b\phi  f)^2  -  2\phi b g_1^*(\dd u^p, \dd f) + 2b^2\phi  =
 (G_1)_{pp}  - K_1 \tilde (u^p)^2,
 $$
 as stated. \end{proof}
 
To prove Theorem \ref{thm:F_casimirs}, we will proceed by induction. In particular, we will use $g = g_1 + \tfrac{1}{f^2} g_2$ as a {\it new} metric $g_2$.  For this purpose, we will need an analog of the function $\phi$ for $g$.   
 Assume that $K_1 =0 $, then  $g$ is flat and there exists a (unique) function $\phi_g$ such that for each flat coordinate $w$ from Proposition \ref{prop:bols2} we have
$$
g^*(\dd \phi_g, \dd w) = w \quad\mbox{and}\quad g^*(\dd\phi_g, \dd\phi_g) = 2\phi_g.
$$ 

\begin{Proposition}\label{prop:bols3}
Let $\phi_{g_1}$ be the function on the first block with above properties w.r.t. $g_1$.  Then this function will still satisfy these properties w.r.t. $g$, that is, $\phi_g = \phi_{g_1}$.
\end{Proposition}

\begin{proof} These properties are obviously satisfied for $w= u^i$, $i=1,\dots, p-1$.  We also have
$g^*(\dd\psi, \dd\psi) = g_1^*(\dd\psi, \dd\psi) = 2\psi$.

So we only need to compute $g^*(\dd\psi,  \dd(fv^i))$ and $g^*(\dd\psi, \dd \tilde u^p)$. We have
 $$
 g^*(\dd\phi_{g_1},  \dd(fv^i)) = g^*(\dd\phi_{g_1}, v^j \dd f + f\dd v^j) = v^j g^*(\dd\phi_{g_1},  \dd f) = f v^j.
 $$
 Similarly, 
$$
\begin{aligned}
g^*(\dd\phi_{g_1},  \dd \tilde u^p)= g^*(\dd\phi_{g_1},  \dd u^p - b f \dd \phi - b\phi \dd f)&=\\
g^*(\dd\phi_{g_1},  \dd u^p) - b\phi g^*(\dd\phi_{g_1}, \dd f) &= 
u^p - b\phi f = \tilde u^p,
\end{aligned}
$$
as required.  \end{proof}

Thus, we have proved everything we need in the case of two blocks.  One can easily see that in the case when $\mathsf F$ contains only two vertices with one edge between them,  the (generalised) flat coordinates for $g$ given by Propositions \ref{prop:bols1} and \ref{prop:bols1} coincides with those from Theorem \ref{thm:F_casimirs}.   The general case can now be obtained by applying the above Propositions to an arbitrary graph $\mathsf F$.

Indeed Propositions \ref{prop:bols1}, \ref{prop:bols2} and   \ref{prop:bols3}  allow us to reconstruct flat coordinates for $g$ by adding blocks step-by-step.  Notice that at each step,  flat coordinates for every ``intermediate'' metric related to any connected subgraph of $\mathsf F$ will be defined via the same procedure.  The three steps from Section \ref{sect:3steps} can be understood as adaptation of Propositions  \ref{prop:bols1} and \ref{prop:bols2}   to our more specific situation.   This step-by-step reconstruction procedure leads to the conclusion of Theorem \ref{thm:F_casimirs}.

\vspace{1ex} 
\noindent{\bf Acknowledgements and Data Availability Statement.}  We thank Ernie Kalnins, Ray Mclenaghan   and the anonimous referee   for their  valuable comments.  The research of V.M.  was supported by  DFG, grants number 455806247 and 529233771, and by 
ARC Discovery Programme, grant DP210100951.   All data generated or analysed during this study are included
in this published article.

\end{document}